\documentclass[10pt,a4paper]{article}
\usepackage{amsmath,amssymb,amsfonts,amsthm}
\usepackage{float,graphicx,color}
\usepackage[all,cmtip]{xy}
\usepackage{tikz}
\usetikzlibrary{matrix}

\usepackage[normalem]{ulem}

\newcommand{\rmd}{\mathrm{d}}

\newcommand{\rmH}{\mathrm{H}}

\newcommand{\rmL}{\mathrm{L}}

\newcommand{\rmW}{\mathrm{W}}

\newcommand{\bbN}{\mathbb{N}}

\newcommand{\bbR}{\mathbb{R}}

\newcommand{\bbU}{\mathbb{U}}

\newcommand{\calP}{\mathcal{P}}

\newcommand{\calT}{\mathcal{T}}

\newcommand{\calV}{\mathcal{V}}

\newcommand{\bs}{{\scriptscriptstyle{\bullet}}}

\newcommand{\id}{\mathrm{id}}

\newcommand{\ts}{\textstyle}

\newcommand{\hs}{\raisebox{1pt}{$\scriptstyle \bigstar$}}

\newcommand{\cleq}{\preccurlyeq}

\newcommand{\Poly}{\calP}
\newcommand{\Pminus}{\Poly_{\!\!-}}

\newcommand{\Whitney}{\rmW}

\newcommand{\test}{\prime}

\newcommand{\beq}{\begin{equation}}
\newcommand{\eeq}{\end{equation}}

\newcommand{\mapping}[4]{
\left\{
\begin{array}{rcl}
\displaystyle #1  &\to& #2\\
\displaystyle #3  &\mapsto & #4
\end{array} \right.
}

\newcommand{\mixedright}[8]{
\
\left\{
\begin{array}{llll}
\forall #1  \in #2 \ & #3  & = #4\\
\forall #5  \in #6 \ & #7  & = #8
\end{array}
\right.
}

\newcommand{\revised}[1]{#1}
\newcommand{\removed}[1]{}
\newcommand{\ignore}[1]{}

\newcommand{\weight}{\revised{\rho^q}}
\newcommand{\weightp}{\revised{\rho^{q'}}}

\newtheorem{theorem}{Theorem}[section]
\newtheorem{proposition}[theorem]{Proposition}
\newtheorem{lemma}[theorem]{Lemma}

\theoremstyle{definition}
\newtheorem{remark}{Remark}[section]

\renewcommand{\subsection}{\paragraph}
\numberwithin{equation}{section}

\newcounter{theopoint}[theorem]
\newcommand{\tpoint}{
\medskip
\stepcounter{theopoint}
\noindent (\roman{theopoint}) 
}

\begin{document}

\title
      {Interpretation of a Discrete de Rham method\\ as a Finite Element System}

\ignore{
\author{Snorre H. Christiansen}
\address{Department of Mathematics, University of Oslo, PO Box 1053 Blindern, NO 0316 Oslo, Norway}
\email{snorrec@math.uio.no.}

\author{Francesca Rapetti}
\address{Université Côte d'Azur, Laboratoire de Mathématiques J.A. Dieudonné, Parc Valrose, F-06108 NICE Cedex 02, France}
\email{francesca.rapetti@univ-cotedazur.fr}
}

\author{Snorre H. Christiansen\thanks{Department of Mathematics, University of Oslo, PO Box 1053 Blindern, NO 0316 Oslo, Norway. email: {\tt snorrec@math.uio.no.}}, Francesca Rapetti\thanks{Université Côte d'Azur, Laboratoire de Mathématiques J.A. Dieudonné, Parc Valrose, F-06108 NICE Cedex 02, France.}}

\date{}

\maketitle

\begin{abstract}
Within the framework of Finite Element Systems (FES) we define a numerical method for computing differential forms, inspired by the Discrete de Rham (DDR) method. The discrete $\rmL^2$ product of DDR is given a new interpretation, as a computable approximation of the true $\rmL^2$ product on local spaces of forms defined by partial differential equations.
Without significantly modifying the numerical method itself, this point of view provides an alternative approach to the analysis. The conformity and consistency properties we obtain are stronger than those previously shown, even in low dimensions. We can also recover some of the other results that have been proved about DDR, from those that have already been proved, in principle, in the general context of FES. Parts of the analysis involve a conjectural hypothesis of stability of constants, with respect to small mesh perturbations, that we detail. We also bring VEM, the Virtual Element Method, into the discussion. For convenience, background material on the FES framework is reviewed.
\end{abstract}

\ignore{


\date{December 19, 2025 and, in revised form, June 17, 2026 and September 23, 2026.}
}


\bigskip


\bigskip


\bigskip

\noindent {\bfseries History:} Submitted to \emph{Mathematics of Computation} on December 19, 2025
(27 pages). Revisions submitted on June 17, 2026 (37 pages) and September
23, 2026 (37 pages). Accepted on October 30, 2026.
\\ This version amended as explained in the last section, \lq\lq Declaration\rq\rq.  
\bigskip

\section{Introduction}
The paper is concerned with defining spaces of differential forms on polytopal meshes, that are suitable for solving Partial Differential Equations (PDEs) numerically, typically within a Galerkin method.

Using techniques introduced in the framework of Finite Element Systems (FES) \cite{Chr08M3AS,ChrMunOwr11} we define a numerical method directly inspired by the Discrete de Rham (DDR) method of \cite{DiPDroRap20, DiPDro23, BonDiPDroHu25}. It shares degrees of freedom with DDR and also with previous methods defined in FES \cite{ChrGil16} and VEM \cite{BeiBreMarRus16}. The exterior derivative can be computed straightforwardly. The crucial point is that the method defines computable discrete $\rmL^2$ products that on the one hand mimick those of \cite{BonDiPDroHu25} but on the other hand have a different interpretation that leads to a different analysis. Precise parallels are drawn in Remark \ref{rem:compare} below.

We interpret the numerical method as being based on discrete spaces defined piecewise by solving certain new PDEs, on each cell of a mesh. The method is \emph{conforming} in the sense of producing subcomplexes of natural complexes of Sobolev spaces on the domain. Even though the solutions to these new local PDEs are not computable, the numerical method itself is ''fully discrete'' because it allows the explicit computation of enough moments against polynomials, from the degrees of freedom. The method is then \emph{consistent} in the sense of Strang's lemmas.  This type of consistency and conformity properties are stronger than those obtained in previous interpretations of DDR, as expressed for instance in \cite{BonDiPDroHu25}, especially the three ''key features'' given as bullet points page 3 of the introduction. See also Remarks \ref{rem:oldinterpret} and \ref{rem:refine} below, regarding an evolution of points of view. The interpretation provided here  embeds the method in the Finite Element System framework, and allows to leverage the results already obtained in that setting. 

\subsection{Previous work on FES} The framework of Finite Element Systems (FES) was introduced in \cite{Chr08M3AS} and expanded upon in \cite[\S 5]{ChrMunOwr11} (this independent section is available on arXiv\footnote{\textsc{S. H. Christiansen,} \emph{Finite element systems of differential forms,} arXiv:1006.4779, 2015.}) and subsequent papers. It is an environment in which to define and analyze sequences of spaces of differential forms, that are finite elements in a broad sense. The underlying mesh can consist of cells (in particular polytopes) and the basis functions can be general (not necessarily based on polynomials). Thus one may call them polytopal discrete de Rham sequences, though this expression is given an a priori different meaning below. These spaces were conforming in the sense that the fields were $\rmL^2$-integrable with exterior derivative also $\rmL^2$-integrable, which enforces that traces, in the sence of pullbacks to faces, coincide (they are ''single-valued''). This encodes a type of partial continuity that is common for the mixed finite elements of electromagnetics and fluid flow \cite{BofBreFor13}, such as those of Raviart-Thomas-Nédélec. The latter were extended to differential forms in \cite{Hip99}. A synthesis with extensions to elasticity is provided in Finite Element Exterior Calculus (FEEC) \cite{ArnFalWin06}.

Since then the framework of FES has been generalized to accomodate higher regularity of the fields \cite{ChrHu18} (as in so-called ''Stokes complexes'' relevant to fluid flow) and other types of fields \cite{ChrHu23} (e.g. symmetric tensors and their associated curvature tensors, as in Riemannian geometry). The definition has a natural expression in the language of categories: a FES is a just a contravariant functor from the (poset) category determined by the mesh, to the category of vector spaces (interpreted as spaces of fields on cells). A short summary is provided in \cite{Chr22}. This definition covers all finite elements described according to Ciarlet's definition, including smooth ones (such as Clough-Tocher and Argyris) and non-conforming ones (such as Morley and Crouzeix-Raviart elements).

This point of view connects finite elements with sheaf theory \cite{Ive86}. Indeed the combinatorial structure of the mesh can be seen as defining a ''site'' with a Grothendieck topology. A contravariant functor from such a site into an algebraic category is a presheaf, which for us represents function spaces than can be both localized (from global to local) and glued back (from local to global). The existence of unisolvent degrees of freedom, which is central to Ciarlet's definition of a finite element (see \cite[\S II.10]{Cia91}), can be interpreted as a softness condition on sheaves (see \cite[\S III.2]{Ive86} for definitions), as detailed in \cite[Proposition 5]{ChrHu23}.

For the purposes of this paper, however, the definition given in \cite{Chr08M3AS} is sufficient. The goal back then was to define a discrete de Rham finite element sequence on polytopal meshes that mimicked the lowest order mixed finite elements that correspond to Whitney forms. Indeed these were known only for simplices (and products of simplices handled by tensor product constructions). For this presentation we will also refer to the more comprehensive presentation in \cite[\S 5]{ChrMunOwr11}. Short presentations can also be found in later publications, in particular \cite[\S 2]{ChrRap16}.

A high order version of the construction of \cite{Chr08M3AS} yielding minimal spaces (sometimes referred to as ''serendipity spaces'') was given in \cite{Chr10CRAS, ChrGil16}. The method of \cite{Chr08M3AS, ChrGil16} was conceptually based on harmonic extensions of differential forms, which is not in general computable. But with the help of, for instance, mesh refinements and known discrete spaces on the refinement, adequate discrete harmonic extensions were defined, yielding ''fully discrete'' numerical methods with algebraic and analytical properties identical with those of standard mixed finite elements.

By modifying the harmonic extension PDE that underlied the above constructions, one can obtain upwinded complexes, to treat convection dominated diffusion problems \cite{Chr13FoCM,ChrHalSor14}. The framework has also been used to give compact descriptions of several known finite element methods \cite{ChrGil16, ChrRap16}. The generalized Whitney forms introduced in \cite{Chr08M3AS} have also been used to analyze Discrete Exterior Calculus (DEC) in \cite{GuzPot25} (compare with \cite{ChrHirHuZhu25}).

\subsection{Other polytopal methods} The Virtual Element Method (VEM), reviewed in \cite{BeiBreMarRus23}, has been developed with similar goals as FES and some similar basic principles, such as definition by recursive extension of the fields, from low dimensional cells. But, crucially, no meshrefinement is used in VEM to define the discretization spaces or the numerical method. The challenge is to reconstruct enough information on the fields from the degrees of freedom, to obtain stable and consistent approximations of the bilinear forms defining the partial differential equation to solve. A connection between the mixed VEM (defined in dimension 3) of \cite{BeiBreMarRus16} and FES (in arbitrary dimension) was made in \cite[\S 2.1]{ChrGil16} and will be used later.  

The Discrete de Rham (DDR) method was introduced in dimension 2 and 3 in \cite{DiPDroRap20} and has since been further developed in several papers. We refer to \cite{BonDiPDroHu25} for an overview and a description of the method valid for differential forms of arbitrary degree in arbitrary dimension. 

The DDR construction was inspired at least in part by Hybrid High Order (HHO) methods. Accordingly, it has been presented as a non-conforming method involving spaces of non-matching fields on cells of all dimension. Non-conformity of the spaces poses major challenges for the analysis of the method, many of which have been addressed in the literature \cite{DiPDro23,DiPDroPit23}. In particular an original notion of ''adjoint consistency'' has been developed.

The close relationship between DDR and VEM has already been pointed out in \cite{BeiDasDiPDro22}. In some versions these methods can be seen to have the same degrees of freedom, and they coincide with those of FES provided in \cite[\S 2.1]{ChrGil16}.

Lowest order numerical methods have been designed on various polytopes for a long time and can be traced back to work of, for instance, Yee and Arakawa in the 1960ies and 1970ies. See \cite{DesHirLeoMar05,BocHym06,LipManSha14,BonErn14, BonErn15} concerning Mimetic Finite Differences (MFD), Discrete Exterior Calculus (DEC) and Compatible Discrete Operators (CDO). Many methods seem to share a common algebraic approach to the differential operators, with links to cochain formalisms in algebraic topology \cite{Mat97}, whereas they have different approaches to the construction of the discrete $\rmL^2$ inner product.

\subsection{Contribution} In this paper, we reinterpret DDR as a conforming FES equip\-ped with computable consistent approximations of the $\rmL^2$ inner products.  The degrees of freedom of \cite[\S 2.1]{ChrGil16} are used.  This shows that these examples of DDR, VEM and FES are very similar, if not identical, in practice. They seem to differ mainly in terms of interpretation and analysis. The main novelty here is just that a new harmonic extension PDE, modifying the one used in \cite[\S 2.1]{ChrGil16}, is introduced to define the discrete spaces.

This point of view allows us to recover some of the results that have been proved for VEM or DDR 
as special cases of results already proved, in their principle, in the general context of FES.  We summarize:
\begin{itemize}
\item Consistency of the discrete $\rmL^2$ scalar product now follows from the fact that the new local PDE allows the explicit recursive computation of $\rmL^2$ products with enough polynomials, from the degrees of freedom.
\item Cohomological properties developed for FES, especially \cite[Proposition 5.15]{ChrMunOwr11}, now carry over to this method. This passage from local exactness (on each cell of the mesh), to globally correct cohomology (on the computational domain) is typical of sheaf-theoretic methods.
\item Uniformly bounded commuting projections, building on \cite{Sch08, Chr07NM}, have been described for FEEC \cite{ArnFalWin06,ChrWin08} and FES \cite[\S 5.3]{ChrMunOwr11}. They enable one to recover Poincaré-Friedrich estimates, which are basic to the stability analysis of mixed methods.
\end{itemize}
In the setting of FES one furthermore obtains refinements in the analysis:
\begin{itemize}
\item Translation estimates are given in \cite[Proposition 5.71]{ChrMunOwr11}, and have applications to the equations of fluid flow.
\item Discrete Rellich compactness, as introduced by Kikuchi for Maxwell \cite[\S IV.19]{Bof10}, holds generally, as proved first in \cite[Corollary 5]{Chr07NM} with applications to the Hodge-Laplace eigenvalue problem on manifolds. See also \cite{ChrWin13IMA} on the relationship of bounded commuting projections to such properties, at an abstract level.
\item Discrete Sobolev injections hold \cite[Propositions 5.69, 5.70]{ChrMunOwr11} and have applications to, for instance, non-linear Maxwell wave equations as in \cite{ChrSch11}. See also \cite{HeHuXu19} on this topic.
\item Fractional order Sobolev estimates are provided in \cite{Chr18}, with applications to the Dirac equation.
\end{itemize}

We insist that in our point of view the method is conforming in the sense that we obtain subspaces of standard Sobolev spaces and that the discrete differential operators (instances of the exterior derivative) are simply the restriction of the ordinary ones. These operators are computable in terms of the chosen degrees of freedom. Taking into account the variational crime arizing from the approximation the $\rmL^2$ inner product is viewed as a separate ingredient in the analysis, addressed with the help of Strang's first lemma.

An example of a FES determined by solving local PDEs and equipped with a computable and consistent bilinear form, has been analyzed in \cite{ChrHal15}. It concerned the Schrödinger equation with a magnetic field (representing an $\bbU(1)$-connection). The local PDEs used on the cells are often related, but not identical, to the global PDE one wants to solve. This aspect can be compared with Trefftz methods which often glue exact local solutions in non-conforming ways \cite{HipMoiPer16}. The analysis of FES methods based on solving local PDEs sometimes involves a partly conjectural argument detailed in Remark \ref{rem:fesperturb} below.

In a sense we have little to add to the design of the DDR and VEM methods themselves. Many of the defining integrations (by parts or not) against cleverly chosen polynomial forms have of course already been exemplified in the literature, though perhaps not in this order and generality. We hope the FES perspective can contribute usefully to the analysis, leading to various systematic convergence proofs. Besides, by providing a clarified rationale for neglecting certain terms and identifying key underlying algebraic facts (such as short exact sequences of complexes and their long exact sequences of cohomology groups), this presentation can ease the extension of the method in various directions. This could be for instance the above mentioned upwinding problem \cite{Chr13FoCM}, or the existing variants of DDR developed for other differential operators \cite{DiPHan24} and regularities \cite{Han23}.

\subsection{Outline} The paper is organized as follows. In section \ref{sec:not} we review notations and provide a reading guide on FES. Section \ref{sec:low} contains the construction in the lowest order case (sequences of spaces containing constants). Section \ref{sec:high} contains the construction in the high order case (sequences of spaces containing polynomials of degree $r$, for any choice of $r \geq 1$). Section \ref{sec:algo} shows how the preceding results can be used to set up a numerical method for a prototype equation and start to analyse it.

\section{Notations and background\label{sec:not}}
 Let $\calT$ be a polytopal mesh of some Euclidean polyhedral domain{, denoted $S$}. When considering a polytope $T$ with a face $T'$, if $u$ is a differential form on $T$, by its trace on $T'$ we always mean the trace in the sense of pullback of $u$ to $T'$ by the inclusion of $T'$ in $T$.

 For each $k \in \bbN$ and $T\in \calT$ we let $X^k(T)$ be the space of differential $k$-forms on $T$ which are in $\rmL^2(T)$ with exterior derivative in $\rmL^2(T)$. The $\rmL^2$ norm is denoted $|\cdot|$ and the natural graph norm on $X^k(T)$ is denoted $\| \cdot \|$. Thus:
\begin{equation}\label{eq:gnorm}
    \| u\|^2 = | u| ^2 + | \rmd u|^2.
\end{equation}
The subspace of $X^k(T)$ consisting of forms with $0$ traces on $\partial T$ is denoted $X^k_0(T)$:
\begin{equation}
  X^k_0(T) = \{ u \in X^k(T) \ : \ u|_{\partial T} = 0\}.
\end{equation}
The subscript $0$ will have the same meaning for other spaces{, introduced later,} as well. Elements on $X^k(T)$ have well defined traces, in negative index Sobolev spaces, only on codimension one faces $T' \subseteq \partial T$. However, if these traces in turn are imposed to be in $X^k(T')$, one can repeat the procedure and follow traces all the way down. One gets a space of forms with improved boundary regularity and some nice properties, such as \cite[Proposition 2.4]{ChrRap16}.
As in \cite[p. 73]{ChrMunOwr11} we denote this space with improved boundary regularity as $\widehat{X}^k(T)$.
All our discrete spaces will have this type of extra boundary regularity. In other words they will be subspaces of $\widehat{X}^k(T)$ and taking arbitrary traces will be well defined.

The space of Whitney $k$-forms (which could also be attributed to Weil and de Rham) was originally presented in terms of partitions of unity, such as barycentric coordinates, but it can also be presented in terms of the Koszul operator \cite[\S 3.2]{ArnFalWin06}. For each cell $T \in \calT$, we let $b_T$ be its barycentre. On $T$, the Koszul operator $\kappa_T$ is the contraction by the vector field on $T$ given as $x \mapsto x- b_T$ (for $x \in T$). We adopt notations from \cite{ArnFalWin06}. The original lowest order Whitney forms are then $\Pminus^1\Lambda^k(T)$. The high order version are denoted $\Pminus^r\Lambda^k(T)$:
\begin{equation}
\Pminus^r\Lambda^k(T) = \{ u \in \Poly^r \Lambda^k(T) \ : \ \kappa_T u \in \Poly^r \Lambda^{k-1} (T)\}.
\end{equation}
These spaces were defined in \cite{Hip99} in terms of the standard Poincaré null-homotopy operators rather than the Koszul operators and were described in FES terms in \cite{ChrRap16} under the name ''trimmed'' polynomial differential forms.

The Hodge star operator is denoted $\hs$. It maps from $k$-forms on $T$ to $(d-k)$-forms on $T$, when $d = \dim T$. It is characterized by the property that, when $u$ and $v$ are both $k$-forms:
\begin{equation}
  u \wedge \hs v = (u \cdot v)\, \omega,
\end{equation}
where the pointwise scalar product is denoted $(\cdot)$ and the volume form on $T$ associated with the scalar product and the orientation is denoted $\omega$. For instance $\hs 1 = \omega$.
The formal adjoint of $\rmd$ with respect to the $\rmL^2$ scalar product is denoted $\rmd^\star$. We have $\rmd^\star = \pm \hs \rmd \hs$.

A third point of view on Whitney forms, developed in \cite{Chr08M3AS}, is that they are the solutions to the PDE system:
\begin{equation}
  \rmd^\star \rmd u = 0, \quad \rmd^\star u = 0,
\end{equation}
supposed to hold on simplices of all dimensions (equipped with translation invariant metrics). For $0$-forms this reduces to $\Delta u = 0$. For $k$-forms on simplices (or flat cells for that matter) of dimension $k$, this reduces to $\rmd^\star u= 0$, which is a characterization of the constant $k$-forms.

\subsection{Overview of results on FES}  Since the FES framework has not been reviewed recently we now offer a guide to the relevant literature, for the convenience it might provide the reader.
\begin{itemize}
\item In \cite{Chr08M3AS} the FES framework was introduced as auxiliary discrete spaces inside which analogues of (lowest order) Whitney forms could be defined as locally harmonic forms.
\item Reference \cite{Chr09AWM} is mostly expository.  In \S 3.1 some background in algebra is provided (the five lemma, the snake lemma and the long exact cohomology sequence, the Künneth theorem on tensor products). In \S 3.2 basic concepts in differential geometry are reviewed. This algebraic and differential geometric background is then applied to study FES. The motivation was the construction of Galerkin methods for Maxwell's equations, complementing \cite{Hip02,Jol03,Mon03} on some points, particularly in regards to general meshes and basis functions.

A ''geometric decomposition'' in terms of ''bubble spaces'', in the sense of \cite{ArnFalWin09}, is given in \S 3.3.2 (Proposition 3.14) as a byproduct of discretely harmonic forms.

A first study of tensor products of FES is provided in \S 3.3.3 and further developed in later publications. This theory covers standard elements as in \cite[\S 3]{ArnBofBon15} and splines as in \cite{BufRivSanVaz11}. In fact, one can start with any space of functions on an interval that contains a non-trivial constant function and one with different values at the endpoints. The latter function can be for instance an exponential. 

\item In \cite[\S 5]{ChrMunOwr11} the framework was further developed in its own right, to account for the mixed finite element spaces synthesized in FEEC. In \S 5.3 commuting quasi-interpolators deduced from smoothings were studied. In \S 5.4 stability properties were deduced.

\item In \cite{Chr13FoCM} it is illustrated that by modifying the definition of harmonicity one can get upwinded complexes of differential forms, useful for convection diffusion problems. The above mentioned exponential function allows one to recover some tensor product ``exponentially fitted'' methods. The scalar convection diffusion problem is analyzed in \cite{ChrHalSor14} in fractional Sobolev spaces (of the type $\rmH^{1/2}$).
  
\item In \cite{ChrRap16} FES was used to describe the spaces $\Pminus^r\Lambda^k(T)$ of trimmed polynomial differential forms on simplices. In \S 2  a glossary and a seven page summary of FES is provided. Alternative degrees of freedom were proposed. Resolutions were introduced to describe relations in the natural spanning families constructed with barycentric coordinates.

  It is believed that working with canonical generators and relations (instead of selecting a specific basis) and systematically treating restrictions to subcells, could be valuable for implementations, not just the theoretical study. 

\item In \cite{ChrGil16} it was shown that the Serendipity element \cite{ArnAwa14}, the TNT element \cite{CocQiu14} and the trimmed polynomial differential forms are all examples of FES satisfying a minimality condition. A connection with VEM was made explicit in \S 2.1. A general recipy for constructing minimal or ''serendipity'' spaces was given in \S 3.

\item In \cite{ChrHu18} and \cite{ChrHu23} the framework of FES is generalized to account for more types of finite element spaces. In \cite{ChrHu23,Chr22} the connection with category theory and sheaves is described.
\end{itemize}

\begin{remark}[Stability of constants] \label{rem:fesperturb} For the analysis, many arguments involve a straightforward \emph{scalar} scaling (to reference cells of diameter approximately $1$) combined with an appeal to compactness inspired by \cite[Remark p. 64 and proof p. 69]{ArnFalWin06}. An optimal formulation of such compactness arguments remains to be delineated in the context of FES. It is relevant to the control of constants appearing in inf-sup conditions (in the sense of Babuska and Brezzi), consistency estimates (in the sense of Strang) and approximation errors (Bramble-Hilbert type results). This type of result is sometimes phrased in terms of the dependence of certain constants (e.g. in Sobolev injections, trace theorems and Poincaré inequalities) upon the domain (which is typically an element or a macro-element consisting of the first layer of cells surrounding a given cell). For polynomial finite elements that are invariant under affine transformations, such estimates are usually straightforwardly obtained by mapping back and forth to reference elements. The elements presented here, defined by PDEs, do not have such a strong invariance property: they are invariant only with respect to translations, rotations and scalings $x \mapsto h x$ for scalar $h$. Composite elements based on refinements constitute an intermediate level of difficulty. What seems true, quite generally, is that the discrete spaces depend continuously (even smoothly) on mesh parameters (such as the positions of the vertices of the polytopes and possible extra information such as choices of inpoints), and that locally (for a cell of diameter around $1$, with given combinatorial type but with variable metric properties) the space of parameters is compact. More formally one would like to confirm, in particular, that the geometric configurations for the polytope $T$ constitute a finite dimensional manifold and the spaces $\hat{X}^k(T)$, as well as the constructed discrete subspaces, define Banach vector bundles over this manifold. In this view, the case of polynomials on simplices is simpler because the manifold of allowed configurations for the polytope $T$ is affine and the bundle built from the spaces $\Poly^r \Lambda^{k} (T)$ is flat, locally trivialized by the pullback associated with the canonical affine maps between simplices.  To sum up, on the one hand the property of stability of constants with respect to small perturbations of mesh parameters seems non-trivial, but on the other hand we are not aware of any interesting situation in which it fails.
\end{remark}
We acknowledge that other types of arguments have been developed in the literature on polytopal methods and that they have allowed to conclude directly in some important cases, see in particular \cite{BonErn14,BonErn15} concerning CDO and all the references on VEM and DDR provided in \cite[\S 7]{BeiDasDiPDro22}.

\subsection{Comparison of some early references} The inspiring \emph{IMA Hot Topics work\-shop}, held in May 2004 in Minneapolis, enabled many compatible discretization methodologies to be compared \cite{ArnBocLehNic06}. The framework of FES was introduced in \cite{Chr08M3AS}, submitted in December 2005. It is contemporary with the FEEC reference \cite{ArnFalWin06}. It is also contemporary with \cite{BreLipShaSim07}. The latter can be seen as an intermediate step in the evolution of Mimetic Finite Differences \cite{BocHym06} into Virtual Element Methods, deemed ''brand new'' in \cite{BeiEtAl13}. Complexes were introduced in VEM later, in \cite{BeiBreMarRus16} (in dimension 3). The two last spaces in the complexes of composite elements defined in \cite{Chr08M3AS} can be interpreted as a generalization of the construction of \cite{KuzRep05}. These composite complexes were rediscovered in \cite{PasVas08} (in dimension 3), submitted July 2007.

Though it was clear that the first VEM methods were based on using extensions (from boundaries of cells to their interior) as an efficient construction tool on polytopal meshes\footnote{At the \emph{CMA Workshop: Compatible Discretizations for Partial Differential Equations}, held in September 2005 in Oslo, one of the inventors of VEM presented the case of the Laplacian on a square grid and concluded characteristically by ''leaving all the trivial generalizations to the audience''.} and that FEEC on the other hand emphasized complexes in arbitrary dimension, it seems that \cite{Chr08M3AS} was the first to axiomatize the combination of polytopal meshes in arbitrary dimension and families of discrete spaces that are closed under traces (from cells to subcells) and forming complexes under the exterior derivative. The definition included, but was not restricted to, composite elements on refinements. PDE based techniques were also illustrated.

The definitions were inspired by \cite[\S II.1]{Har77} and \cite[\S II.10]{BotTu82}. The sheaf theoretic connection and functorial perspectives were developed further in later publications, partly in view of their increased role in other branches of applied mathematics \cite{Ghr14}.

The statement, made in \cite[p. 46]{BonDiPDroHu25}, that FES ''only provides algebraic re\-sul\-ts'' seems worthy of two comments. First, the spirit of FES has been to develop enough algebra to ensure that some known analytic arguments go through with only minor modifications. Second, new analytic arguments have been developed in this framework. In the context of complexes of differential forms, the technique of combining interpolators deduced from degrees of freedom with smoothing operators based on convolutions, for the purpose of convergence proofs, was introduced in \cite{Chr07NM}, submitted August 2005. The technique was polished for FES in \cite[\S 5.3]{ChrMunOwr11}.

\begin{remark}[Commuting projections]
More recent constructions of \emph{local} boun\-ded commuting projections \cite{FalWin14, ArnGuz21} carry over to FES in a very natural way. The approach of \cite{FalWin14} extends to FES of differential forms and the approach of \cite{ArnGuz21} extends to general FES as in \cite{ChrHu23} (the latter approach does not appeal to a wedge product). In fact the FES framework seems to identify exactly the algebraic properties of discretization spaces needed for these constructions. A general analysis of stability and compactness properties of these methods remains, in the context of FES, to be fleshed out. See \cite{ErnGuzPotVoh25} for the state of the art.
\end{remark}

\subsection{Functorial definition of a FES}
Since the first publication, the definition of FES has been generalized several times. The final definition, as expressed in \cite{ChrHu23,Chr22}, is simple. It is more an alternative to Ciarlet's definition of a Finite Element (FE) than a concrete method.

We fix a cellular complex $\calT$. The inclusion relation on $\calT$ defines a partial order. A partial order is an example of a preorder and any preordered set is a category \cite[p. 11]{McL98}. We denote this one by $\widehat{\calT}$. We then define a FES to be a contravariant \emph{functor} (in the sense of \cite[\S I.3]{McL98}) from the category $\widehat{\calT}$ to the category of real vector spaces.

An (infinite dimensional) example, is the functor $\widehat{X}^k$ defined above and considered for fixed $k$. As a functor it collects both the data consisting of the spaces $\widehat{X}^k(T)$, where $T \in \calT$, and the trace maps $\widehat{X}^k(T) \to \widehat{X}^k(T')$, defined for $T, T' \in \calT$, every time $T' \subseteq T$. The contravariance is that when $T' \subseteq T$, the trace operator goes the other way.

To express that a given FES $A$ consists of $k$-forms and that the trace operators defined by $A$ really are the traces in the sense of pullback of differential forms, one can use the notion of \emph{natural transformation} (in the sense of \cite[\S I.3]{McL98}) between functors. Indeed it suffices to require that there is a natural transformation from $A$ to $\widehat{X}^k$, given by injective maps $A(T) \to \widehat{X}^k(T)$ for each $T \in \calT$. The naturality is the condition that these injections commute with respective trace operators.

To express that one has a sequence $A^k$ of FES  for $k \in \bbN$, one requires that one has a natural transformation from $A^k$ to $A^{k+1}$ for each $k$ and that the composition of two consecutive transformations is $0$. If one requires, as above, that each $A^k$ naturally injects in $\widehat{X}^k$, one recovers the situation described explicitely in \cite{Chr08M3AS}. Replacing $\widehat{X}^k$ by other choices, one can describe other types of continuities on differential forms, as in \cite{ChrHu18}, or other types of tensors, as in \cite{ChrHu23}.

The global space (used in Galerkin methods) defined by the local data collected by a FES $A$ is the \emph{inverse limit} of the functor (in the sense of \cite[\S V]{McL98}). This definition captures that the global fields are families of local fields (attached to cells) that are continuous in the sense of having matching traces on subcells (interfaces).

If one wants to recover the global geometry (essentially the domain) from local data, one can consider that the local data is collected in a covariant functor from $\widehat{\calT}$ into a category of manifolds with boundary, and then take its \emph{direct limit}. The category of manifolds should contain objects describing the geometrically realized cells (or, as a variant, their tangent spaces). One can take for instance the category of Lipschitz manifolds or that of piecewise smooth manifolds. The direct limit itself will not in general be a manifold, but rather a topological space (with extra properties such as being a CW space). This generality is essential for some proofs that work with skeletons of the mesh, such as \cite[Theorem 3.1 p. 745]{Chr08M3AS}.

For the global space to be a finite element space in the traditional sense, in particular with a ''local basis'', one imposes a \emph{softness} condition. If the previous definitions are just ''abstract nonsense'', this is the technically subtle point. In Ciarlet's definition of a FE, it is captured by the unisolvence of degrees of freedom \cite[\S II.10]{Cia91}. It can also be expressed as a dimension count involving bubble spaces or through the possibility to extend fields from boundaries of cells to their interior. In the context of PDE based FES, the last property is central and was the perspective of \cite{Chr08M3AS}. For a proof of the equivalence of these three properties we refer to \cite[\S 3.2]{ChrHu23}.

At this point one still has a very general definition. It can be specialized to different concrete computational methodologies, such as polynomial based FEEC on simplices or PDE based VEM on polytopes. On the other hand, the softness condition is too strong for some applications to Spline methods.

\subsection{Choice of language}
The remainder of this paper seeks to illustrate one way in which DDR can enter the general picture, without relying on the category theoretic jargon. What's more, even though we adopt the FES point of view, we first present our results without relying on its specific definitions. On the other hand \cite[\S 2]{ChrRap16} contains a glossary and a short presentation that is more than sufficient for this article.

\section{Lowest order case\label{sec:low}}
Let $T$ be a cell. The case $\dim T =0$ need not be detailed, so we suppose that $\dim T = d > 0$.
In this section we define $W^k(T)$ as a slight modification of $\Pminus^1 \Lambda^k(T)$, namely, that for, $k= 0$ we impose that the elements of $W^0(T)$ should have integral $0$.

\subsection{Definition of $Z^k_0(T)$} We define:
\begin{align}
  Z^k_0(T) & = \{u \in X^k_0(T) \ : \ \forall w \in W^{d-k}(T) \quad \int_T u \wedge w = 0\}. 
\end{align}

\begin{proposition}\label{prop:zexact}
  The spaces $Z^\bs_0(T)$ form a sequence which is exact except at index $k=\dim T$, where the integral determines an isomorphism from the cohomology group to $\bbR$.
\end{proposition}
\begin{proof}
  We have a short exact sequence of complexes:
\begin{equation}
  0 \to Z^\bs_0(T) \to X^\bs_0(T) \to W^{d-\bs}(T)^\star \to 0.
\end{equation}
The sequence $W^\bs(T)$ is exact, including at indices $0$ and $d$. Hence $Z^\bs_0(T)$ inherits the cohomology of $X^\bs_0(T)$, by the snake lemma, in its full version. See \cite[Theorem 3.5, p. 600]{Alu09} or \cite[Theorem 3.1, p. 356]{Chr09AWM}.
\end{proof}

The constraint in the definition of $Z^k_0(T)$ can also be expressed as $\rmL^2$ orthogonality with $\hs W^{d-k}(T)$.

\subsection{Definition of $Y^k(T)$} We next define:
\begin{equation}
  Y^k(T) = \{ u \in Z^k_0(T) \ : \ u \perp \rmd Z^{k-1}_0(T) \}.
\end{equation}
Orthogonality is taken with respect to the $\rmL^2$ scalar product (also elsewhere in this paper).
We remark that Proposition \ref{prop:zexact} entails that the following is an isomorphism:
\begin{equation}
\rmd: Y^{k-1}(T) \to \ker \rmd |_{Z^{k}_0(T)},
\end{equation}
for $k < d$, and that, for the case $k =d$, we also have the isomorphism:
\begin{equation}
\rmd: Y^{d-1}(T) \to \ker \ts \int_T(\cdot) |_{Z^{d}(T)}.
\end{equation}
The purpose of $Y^k(T)$ is to be used in the following definition of $B^k(T)$.

\subsection{Definition of $B^k(T)$} We define $B^k(T)$ to be the subspace of $X^k(T)$ consisting of fields $u \in X^k(T)$ such that the following system of PDEs holds:
\begin{align}
  \int_T \rmd u \cdot \rmd u' + \int_T \rmd u' \wedge \kappa_T a + \int_T u' \wedge \kappa_T b + \int_T u' \cdot \rmd v & = 0,\\
  \int_T \rmd u \wedge \kappa_T a' =0 ,\ \int_T u \wedge \kappa_T b' & = 0 , \label{eq:constraintab}\\
  \int_T u \cdot \rmd v' & = 0,
\end{align}
for all $u' \in X^k_0(T)$, for some $a \in \Poly^0\Lambda^{d-k}(T)$, some $b \in \Poly^0\Lambda^{d-k +1}(T)$, and some $v \in Y^{k-1}(T)$,  for all $v' \in Y^{k-1}(T)$, all $a' \in \Poly^0\Lambda^{d-k}(T)$ and all $b' \in \Poly^0\Lambda^{d-k +1}(T)$. 

Notice that $a', b'$ and $v'$ enforce constraints on $u$, and that $a, b, v$ can be interpreted as the corresponding Lagrange multipliers. We refer to $a'$ and $b'$ as the finite dimensional constraints.

The following remark gives some motivation for the definition of $B^k(T)$.

\begin{remark}[Some previous PDE based spaces]
Removing the finite dimensional space of constraints represented by $a,a'$ and $b,b'$, and allowing general $v,v'$, we obtain the harmonic extension used in \cite{Chr08M3AS}. Indeed the equations considered there were:
\begin{align}
  \rmd^\star \rmd u + \rmd v & = 0,\\
  \rmd^\star u & = 0.
\end{align}
The Lagrange multiplier $v$ was always $0$. This can be interpreted as solving a Hodge-Laplace problem. Indeed, with $-\Delta = \rmd^\star \rmd + \rmd \rmd^\star$, the above system is equivalent to:
\begin{equation}
  \Delta u = 0,\ \textrm{ and }\ (\rmd^\star u) |_{\partial T} = 0.
\end{equation}
In both cases the boundary condition is that $u|_{\partial T}$ is given (in the sense of pullbacks).

In \cite{Chr08M3AS}, discrete variants obtained from weak formulations on suitable concrete finite dimensional spaces, were also used to make the method ''fully discrete''.

To sum up, the definition of $B^k(T)$ involves a variant of the Hodge-Laplace equation deduced for certain boundary conditions, which is modified by some polynomial constraints. The solution must satisfy $\rmL^2$ orthogonalities with spaces of polynomials and the PDE in strong form must be satisfied modulo these spaces. The crucial role played by these constraints will appear in Proposition \ref{prop:momconst}. They help make the method ''fully discrete'' by another strategy than the one of \cite{Chr08M3AS}.
\end{remark}

We now detail the special cases $k=0$ and $k=d$ of the definition of $B^k(T)$.

\noindent For $k= 0$ we are left with the following system:
 \begin{align}
  \int_T \rmd u \cdot \rmd u' + \int_T \rmd u' \wedge \kappa_T a = 0,\\
  \int_T \rmd u \wedge \kappa_T a' = 0.
\end{align} 

\noindent For $k=d$, we are left with:
\begin{align}
  \int_T u' \wedge \kappa_T b + \int_T u' \cdot \rmd v = 0,\\
  \int_T u \wedge \kappa_T b' = 0 ,\\
  \int_T u \cdot \rmd v' = 0,
\end{align}
Notice that $\kappa_T b$ is a general element of $W^0(T)$. It has integral $0$ because of the choice of origin in $\kappa_T$. Let $\overline u$ be the average of $u$. Then $u-\overline u$ is in $Y^d(T)$ and has integral $0$, so it is $0$. That is $u = \overline u$. In other words $B^d(T)$ consists of the constant $d$-forms.

\subsection{Definition of $A^k(T)$} Then we define $A^k(T)$ to be the subspace of $X^k(T)$ consisting of fields $u$ whose trace on any face $T'$ of $T$ (including $T' = T$) is in $B^k(T')$. In particular $A^k(T)\subseteq \widehat{X}^k(T)$. Even though $B^k(T)$ will be infinite dimensional for $k < d$, $A^k(T)$ will be finite dimensional. It will be the discrete space underlying the computations.

\subsection{Properties of $A^k(T)$} We make the following claims, which we will justify subsequently:
\begin{itemize}
\item A field in $B^k(T)$, for $k < \dim T$ is uniquely determined by its boundary traces. Moreover the extension problem defined by the above PDE is wellposed.
\item The family of spaces $A^k(T)$ is closed under traces and exterior derivative, making them a FES.
\item The spaces $A^k(T)$ are finite dimensional and have one degree of freedom per $k$-face of $T$. This ensures \emph{softness} of the FES.
\item For each $T \in \calT$, the spaces $A^\bs(T)$ form an exact sequence resolving $\bbR$. Together with softness this ensures \emph{compatibility} of the FES.
\item The space $A^k(T)$ contains the Whitney forms $\Pminus^1 \Lambda^k(T)$. This ensures approximation properties.
\item For any $u \in A^k(T)$ and any $a \in \Poly^0\Lambda^{d-k}(T)$, the integral $\int_T u \wedge a$ can be computed exactly from the degrees of freedom, in a recursive way using the finite dimensional constraints expressed by $a',b'$ in Equation (\ref{eq:constraintab}). This ensures that we can compute with these spaces.
\end{itemize}

According to plan, we first comment on the well-posedness of the PDE. As already stated in connection with the definition of $\widehat{X}^k(T)$ the traces of $X^k(T)$ in the sense of pullback to $\partial T$ are well defined in some negative index Sobolev spaces. The exact range will be denoted $\partial X^k(T)$. The range $\partial X^k(T)$ has been characterized \cite{MitMitSha08}, but all we need to know is that it contains the differential $k$-forms on $\partial T$ that are $\rmL^2$ with exterior derivative in $\rmL^2$. We will use the following.

\begin{proposition} \label{prop:bwellposed} On a cell $T$ of dimension $d$ and for $k < d$, the above system of PDEs for $u \in B^k(T)$ is wellposed, for given boundary data in $\partial X^k(T)$ for $u$.
\end{proposition}

\begin{proof}
\tpoint As is standard, we first reformulate the boundary value problem with homogeneous right side as one with homogeneous boundary condition and nonzero right side.

Extend the boundary data to an arbitrary field $\tilde u \in X^k(T)$ (for instance by ordinary harmonic extension). Then write the unknown $u \in X^k(T)$ as $u = \tilde u + \hat u$ with $\hat u \in X^k_0(T)$. We now consider the system:
\begin{align}
  \int_T \rmd \hat u \cdot \rmd u' + \int_T \rmd u' \wedge \kappa_T a + \int_T u' \wedge \kappa_T b + \int_T u' \cdot \rmd v & = - \int_T \rmd \tilde u \cdot \rmd u',\\
  \int_T \rmd \hat u \wedge \kappa_T a' = - \int_T \rmd \tilde u \wedge \kappa_T a'  ,\ \int_T \hat u \wedge \kappa_T b' & = - \int_T \tilde u \wedge \kappa_T b' ,\\
  \int_T \hat u \cdot \rmd v' & =  - \int_T \tilde u \cdot \rmd v' ,
\end{align}

Since $k < d$ we can now integrate by parts the constraints, as in:
\begin{align}
\int_T \rmd \hat u \wedge \kappa_T a' + \int_T \hat u \wedge \kappa_T b' & = \int_T \hat u \wedge ( (-1)^{k+1} \rmd \kappa_T a' + \kappa_T b'),\\
&=  \int_T \hat u \wedge ((-1)^{k+1}(d-k) a' + \kappa_T b') ,
\end{align}
For $k = 0$ there is no $b'$ term.
We recognize the general expression of a $w' \in \Whitney^{d-k}(T)$:
\begin{equation}\label{eq:wprime}
  w' = (-1)^{k+1}(d-k) a' + \kappa_T b',
\end{equation}
uniquely corresponding to $a'$ and $b'$.

We can rewrite the system for  $\hat{u} \in X^k_0(T)$ as, for all $u' \in X^k_0(T)$:
\begin{align}
  \int_T \rmd \hat u \cdot \rmd u' + \int_T u' \wedge w + \int_T u' \cdot \rmd v & = - \int_T \rmd \tilde u \cdot \rmd u', \label{eq:usad1}\\
  \int_T  \hat u \wedge w' & = - \int_T \rmd \tilde u \wedge \kappa_T a' - \int_T \tilde u \wedge \kappa_T b', \label{eq:usad2}\\
  \int_T \hat u \cdot \rmd v' & =  - \int_T \tilde u \cdot \rmd v' , \label{eq:usad3}
\end{align}
for some $w \in W^{d-k}(T)$ and $v \in Y^{k-1}(T)$, and for all $w' \in W^{d-k}(T)$ (written as (\ref{eq:wprime})) and $v' \in Y^{k-1}(T)$.

\tpoint As a preliminary check, we look at the homogeneous problem (zero right sides). The second equation says that $\hat u \in Z^k_0(T)$ and the third one then says that $\hat u\in Y^k(T)$. The first equation with $u' = \hat u$ gives $\rmd \hat u =0$. This shows that $\hat u = 0$. The homogeneous problem thus has only the trivial solution.

\tpoint The proof of the well-posedness in the PDE sense, is based on the theory of inf-sup conditions and is contained in the Lemmas \ref{lem:infsup} and \ref{lem:coerker}, below.
\end{proof}

We begin with the following remark.
\begin{remark}\label{rem:wpy} We have:
   \begin{equation}
     (\hs W^{d-k}(T)) \cap \rmd Y^{k-1}(T) = 0.
   \end{equation}
Indeed it holds, more generally, that for $u \in Z^{k}_0(T)$ we have $u \perp \hs W^{d-k}(T)$ by definition.
\end{remark}

\begin{lemma}[Inf-Sup condition]\label{lem:infsup} Fix $k$ and $T$. Suppose that $E$ is a finite dimensional subspace of $X^k(T)$  and that $E \cap \rmd Y^{k-1}(T) = 0$. Then:
\begin{equation}
  \inf_{v \in \rmd Y^{k-1}(T), w \in E} \sup_{u\in X^k_0(T)} \frac{|\langle u, v \rangle + \langle u, w \rangle|}{\| u\| (\| v \| + \| w\|)} >0.
\end{equation}
\end{lemma}
\begin{proof}
  If not, construct sequences $v_n \in \rmd Y^{k-1}(T)$ and $w_n \in E$ such that $\|v_n\| +\|w_n\| = 1$ and:
  \begin{equation}
    \sup_{u\in X^k_0(T)} \frac{|\langle u, v_n \rangle + \langle u, w_n \rangle|}{\| u\|} \to 0.
  \end{equation}
  Extracting subsequences we suppose futhermore that $v_n \to v$ weakly in $X^k_0(T)$ and that $w_n \to w$ strongly in $E$ (by finite dimensionality of $E$).

  From the above, we get that $v + w = 0$. Also $\rmd Y^{k-1}(T)$ is closed in $X^k_0(T)$, so $v\in \rmd Y^{k-1}(T)$. Therefore $v= 0$ and $w= 0$. Thus we conclude that $w_n \to 0$ strongly. From that we get:
  \begin{equation}
    \sup_{u\in X^k_0(T)} \frac{|\langle u, v_n \rangle|}{\| u\|} \to 0.
  \end{equation}
  But this is false, since we can use $u = v_n$ in this $\sup$.
  \end{proof}

\begin{lemma}[Coercivity on the kernel]\label{lem:coerker}
  There is $C>0$ such that for all $u \in Z^k_0(T)$, such that $u \perp \rmd Y^{k-1}(T)$ (w.r.t. $\rmL^2$ scalar product):
\begin{equation}
   |u| \leq C | \rmd u|.
\end{equation}
\end{lemma}
\begin{proof}
  If not choose a sequence $u_n \in Z^k_0(T)$, such that $u_n \perp \rmd Y^{k-1}(T)$, $|u_n| = 1$ and $|\rmd u_n| \to 0$. Extracting subsequences we suppose that $(u_n)$ converges weakly in $\rmL^2$ to some $u$. Then $u \in Z^k_0(T)$ and $\rmd u =0$ (by weak continuity of $\rmd$) and $u \perp \rmd Y^{k-1}(T)$. Hence $u=0$.

  We now show that the convergence $u_n \to 0$ is actually strong in $\rmL^2$. For that purpose we control $\rmd^\star u_n$ in $\rmL^2$. We decompose, with respect to the $\rmL^2$ product:
  \begin{equation}
    X_0^{k-1}(T) = Z_0^{k-1}(T) \oplus F.
  \end{equation}
  Any $v\in X_0^{k-1}(T)$ such that $v \perp \rmd X^{k-2}_0(T)$, can now accordingly be written $v = v' + w$ with $v' \in Z^{k-1}_0(T)$ and $w \in F$. We have:
  \begin{equation}
   \sup_{v \in X_0^{k-1}(T)} \frac{|\int_T u_n \cdot \rmd v|}{| v|} \leq \sup_{w \in F} \frac{|\int_T u_n \cdot \rmd w|}{|w|}.
  \end{equation}
  Since $F$ is finite dimensional and $u_n\to 0$ weakly in $\rmL^2$, we get that $\rmd^\star u_n \to 0$ strongly in $\rmL^2$.

  Since the subspace of $X^k_0(T)$ of consisting of fields $u$ with $\rmd^\star u \in \rmL^2(T)$, with its natural topology, is compactly injected in $\rmL^2(T)$, $(u_n)$ converges strongly in $\rmL^2$ norm to $0$. But this contradicts $|u_n| = 1$.
  \end{proof}

We continue where we left in the proof of Proposition \ref{prop:bwellposed}.
\begin{proof}(continued)
  We use the theory of inf-sup conditions for saddle point problems as for\-mu\-lated in e.g. \cite[Theorem 4.3 p. 132]{Bra07} or \cite[Theorem 4.2.3 p. 228]{BofBreFor13}.
\begin{itemize}    
\item We apply Lemma \ref{lem:infsup} with $E = \hs W^{d-k}(T)$ using Remark \ref{rem:wpy}. This shows that the saddle point problem (\ref{eq:usad1},\ref{eq:usad2},\ref{eq:usad3}) satisfies the Brezzi compatibility condition. 

\item Lemma \ref{lem:coerker} shows that the saddle point problem (\ref{eq:usad1},\ref{eq:usad2},\ref{eq:usad3}) satisfies the coercivity on the kernel.
\end{itemize}
These are the two conditions needed to conclude.
\end{proof}

\begin{proposition}[FES] The family of spaces $A^k(T)$ is closed under traces and the exterior derivative. 
\end{proposition}
\begin{proof}
\tpoint  Closedness under traces follows from the definition of $A^k(T)$, that already took into account all the traces.

\tpoint  Closedness under the exterior derivative will follow from the fact that if $u \in B^k(T)$ then $\rmd u \in B^{k+1}(T)$. Suppose then that $u \in B^k(T)$. We have for all $u' \in X^k_0(T)$:
  \begin{equation}
\int_T \rmd u \cdot \rmd u' + \int_T u' \wedge w + \int_T u' \cdot \rmd v = 0.
  \end{equation}
  If now $u'\in Y^k(T)$, the second and third terms are 0. This gives the last constraint on $\rmd u$.

The two finite dimensional constraints are trivially satisfied.
The top equation is also trivially satisfied, with zero Lagrange multipliers.
\end{proof}

\begin{proposition}[Degrees of Freedom]\label{prop:ldof}
On $A^k(T)$ the integrals on $k$-faces of $T$ constitute unisolvent degrees of freedom.
\end{proposition}
\begin{proof}
Any given $u \in A^k(T)$ is determined by its codimension one traces, that can be followed all the way down to faces of dimension $k$, where $u$ must be constant, hence determined by its integral. Conversely from any values of these integrals, a $u \in A^k(T)$ can be reconstructed by recursive extension of the corresponding constant forms on $k$-cells.
\end{proof}

\begin{proposition}[Local sequence exactness]
The sequence $A^\bs(T)$ is exact (more precisely it resolves $\bbR$).
  \end{proposition}
\begin{proof}
  The degrees of freedom provide a cochain isomorphism with the cellular cochain complex, which is known to be exact on any given cell.
\end{proof}

\begin{proposition}[Approximation]
  The spaces $A^k(T)$ contain $\Pminus^1 \Lambda^k(T)$.
\end{proposition}

\begin{proof}
Whitney forms satisfy $\rmd^\star u = 0$ and $\rmd^\star \rmd u = 0$, so the crucial points are the finite dimensional constraints.

It was proved in \cite[Preprint, Proposition 3.9]{Chr07NM} that:
\begin{equation}
  \wedge : \Pminus^r \Lambda^k(T) \times \Pminus^s \Lambda^l(T) \to \Pminus^{r+s}\Lambda^{k+l}(T).
\end{equation}
A simpler proof based on the Koszul operator, rather than barycentric coordinates, was then provided in \cite[Equation (3.16)]{ArnFalWin06}. Based on this result, if $u\in \Pminus^1 \Lambda ^k(T)$ and $b' \in \Poly^0 \Lambda^{d-k+1}(T)$ we have that $u \wedge \kappa_T b$ is a linear $d$-form. Its integral can therefore be obtained from evaluation at the barycentre of $T$. Since furthermore $\kappa_T$ was constructed with respect to this point, we get that $\int_T u \wedge \kappa_T b' = 0$. Similarly $\int_T \rmd u \wedge \kappa_T a' = 0$ for $a' \in \Poly^0 \Lambda^{d-k}(T)$. 
\end{proof}

In particular, if $T$ is a simplex, we have  $A^k(T) = \Pminus^1 \Lambda^k(T)$ by dimension count.

\begin{proposition}[Computation of moments]\label{prop:momconst}
For $u \in A^k(T)$, the integrals $\int_T u \wedge a'$ against constant $a' \in \Poly^0 \Lambda^{d-k} (T)$ can be computed recursively from the degrees of freedom.
\end{proposition}

\begin{proof}
  Let $T$ be a cell of dimension $d$. The case $k = d$ is immediate, so we suppose $k < d$.

  We integrate by parts the first constraint (as before, but now with a boundary term):
  \begin{equation}
    \int_T \rmd u \wedge \kappa_T a'  = \int_{\partial T} u \wedge \kappa_T a' + \int_T u \wedge (-1)^{k+1}(d-k) a'. 
  \end{equation}
  The left side is zero by definition. On the right side, the second term is of the form that we would like to compute.

  We are left with the problem of evaluating the first term on the right side. It will be decomposed into integrals on the $(d-1)$ dimensional faces $T'$ of $T$. That is, we require terms of the form:
  \begin{equation}
    \int_{T'} u \wedge \kappa_T a'.
  \end{equation}
We use that $(\kappa_T -\kappa_{T'})a'$ is constant on $T'$ and that by definition:
\begin{equation}
    \int_{T'} u \wedge \kappa_{T'} a' = 0.
\end{equation}
The boundary term has been reduced to integrals of $u$ wedged against constants.

This method can be pursued, decreasing the dimension by one at each step, all the way down to the case where the integrals against constants are actually degrees of freedom. 
  \end{proof}

For a given $u \in A^k(T)$ one can then compute the $\rmL^2(T)$ projection of $u$ onto constants $\Poly^0 \Lambda^k(T)$, from its degrees of freedom. We denote it as $\gamma_T u$.  In practice one computes $\gamma_{T'} u$ starting with $\dim T' = k$ and increasing dimension of $T'$ by one at each step.

In \S \ref{sec:algo} we sketch how $\gamma_T$ is used as a crucial ingredient in the design of some numerical methods, that are amenable to an analysis based on Strang's first lemma \cite[Theorem 26.1 p. 192]{Cia91}.
  
The preceding results concerned individual cells $T \in \calT$. We have proved all the ingredients that go into the following global result:
\begin{proposition}
  The global spaces $A^\bs(\calT)$ constitute a complex under the exterior derivative, and, in the canonical degrees of freedom, the matrix of the exterior derivative is the incidence matrix of the mesh, collecting relative orientations of cells.
\end{proposition}

\section{High order case\label{sec:high}}
We fix a polynomial degree $r$ with $r \geq 1$. We will define a method that properly reconstructs $k$-forms that are polynomials of degree $r$. More precisely we will define spaces $A^k(T)$ that contain $\Poly^r \Lambda^k(T)$ and such that the spaces $\Pminus^r \Lambda^{d'-k}(T')$ (with $d' = \dim T'$) constitute unisolvent degrees of freedom. The crucial point, in practice, is that we can compute the more general moments against $\Poly^r \Lambda^{d'-k} (T')$ explicitely (just removing the minus).

\begin{remark}[Link with low order case] To a large extent the low order case treated previously would correspond to $r=0$. However the theory is different on some points, because the spaces $\Pminus^0 \Lambda^k(T)$ (which are trivial  except for $k=0$, yet are the degrees of freedom of the Whitney forms) lack some of the properties of $\Pminus^r \Lambda^k(T)$, for $r \geq 1$.
\end{remark}
  
We set:
\begin{equation}
  U^k(T) = \{u \in \Pminus^{r+1} \Lambda^k(T) \ : \ u \perp \Poly^{r} \Lambda^k(T)\}.
\end{equation}
Orthogonality is with respect to the $\rmL^2$ product. For instance:
\begin{equation}
  U^0(T) = \{u \in \Poly^{r+1} (T) \ : \ u \perp \Poly^{r}(T)\}.
\end{equation}
In particular this space does not contain any non-zero constant function. We also notice that for $d = \dim T$:
\begin{equation}
  U^{d}(T) = 0.
\end{equation}
Then we change the previous definition we have of $W^k(T)$ to the following:
\begin{equation}\label{eq:defw}
  W^k(T) = \rmd U^{k-1}(T) + U^k(T).
\end{equation}
This entails:
\begin{align}
  W^0 & = U^0,\\
  W^d & = \rmd U^{d-1}.
\end{align}

\begin{proposition} We have:
\begin{itemize}  
\item  The decomposition (\ref{eq:defw}) is $\rmL^2$-orthogonal.

\item The following map is bijective:
  \begin{equation}
    \mapping{U^{k-1} \times U^k}{W^k(T),}{(u', u)}{\rmd u' + u.}
    \end{equation}
 
\item The sequence $W^\bs(T)$ is exact, including at indices $0$ and $d$.
\end{itemize}
\end{proposition}

\begin{proof}
\tpoint  Let $u' \in U^{k-1}(T)$ then $\rmd u' \in \Poly^{r} \Lambda^k(T)$ so $\rmd u' \perp U^k(T)$.

\tpoint  We also notice that if $\rmd u' = 0$ then $u' \in \Poly^{r} \Lambda^{k-1}(T)$ so $u' = 0$.

\tpoint  Put now $w = \rmd u' + u$, with $u' \in U^{k-1}(T)$ and $u \in U^k(T)$. If $\rmd w = 0$ then $\rmd u = 0$. Then $u \in \Poly^{r} \Lambda^k(T)$, so $u =0$.  So $w = \rmd u' \in \rmd W^{k-1}(T)$.
\end{proof}

\begin{proposition} We have direct sum decompositions:
  \begin{equation}
    \Pminus^{r+1} \Lambda^k(T) = \Pminus^r \Lambda^k(T) \oplus W^k(T).\label{eq:pmpmw}
  \end{equation}
  \end{proposition}
\begin{proof}
  \tpoint We make a preliminary observation. \\
Referring to \cite[Theorem 3.1]{ArnFalWin06},  we write, for $u \in \Poly^{r+1}\Lambda^k(T)$:
  \begin{equation}
    (\rmd \kappa_T + \kappa_T \rmd)u = (k + r+1) u + Q_r u,\textrm{with } Q_r u\in \Poly^{r}\Lambda^k(T).  
  \end{equation}
  Suppose now that $u \in U^k$ and that $\rmd u \in \Pminus^r \Lambda^{k+1}(T)$. We have:
  \begin{equation}
    (k + r + 1) u = \rmd \kappa_T u + \kappa_T \rmd u - Q_r u.
  \end{equation}
  The three terms on the right side are in $ \Poly^{r}\Lambda^k(T)$. Hence $u=0$.
  
  \tpoint  We now show that the sum is direct.\\
  Let $u' \in U^{k-1}(T)$ and $u \in U^k(T)$ and suppose $w = \rmd u' + u \in  \Pminus^r \Lambda^k(T)$. Then $\rmd u \in \Poly^{r-1}\Lambda^{k+1}(T)$ so $u= 0$. Then $\rmd u' \in  \Pminus^r \Lambda^k(T)$ so $u'=0$.
  
  \tpoint We now show that $\Pminus^{r+1} \Lambda^k(T) \subseteq \Pminus^r \Lambda^k(T) \oplus W^k(T)$.\\
  Suppose that $v \in  \Pminus^{r+1} \Lambda^k(T)$. Write first:
  \begin{equation} \label{eq:decvuv}
  v = u + v', \textrm{with } u \in U^k \textrm{ and } v' \in \Poly^{r}\Lambda^k(T).
  \end{equation}
  Next, as a variant of the above decompositon, write:
  \begin{equation}
  (\rmd \kappa_T + \kappa_T \rmd)v' = (k + r) v' + Q_{r-1} v',\textrm{with } Q_{r-1} v'\in \Poly^{r-1}\Lambda^k(T),
  \end{equation}
  and furthermore:
  \begin{equation}
  \kappa_T v' = u' + v'', \textrm{with } u' \in U^{k-1} \textrm{ and } v'' \in \Poly^{r}\Lambda^{k-1}(T).
  \end{equation}
  We now have:
  \begin{equation}
  (k+r)v' = \rmd u' + \rmd v'' + \kappa_T \rmd v' - Q_{r-1} v'.
  \end{equation}
The first term on the right side is in $\rmd U^{k-1}$ and the three others are in $\Pminus^r \Lambda^k(T)$. Inserting this in (\ref{eq:decvuv}) we are done.
  \end{proof}

We keep the previous definition of $Z^k_0(T)$, given the new definition of $W^{d-k}(T)$, viz.:
\begin{equation}
Z^k_0(T)  = \{u \in X^k_0(T) \ : \ \forall w \in W^{d-k}(T) \quad \int_T u \wedge w = 0\}.
\end{equation}
We remark that $Z^\bs_0(T)$ has the same exactness property as in Proposition \ref{prop:zexact}, since $W^\bs(T)$ is exact. We similarly define $Y^k(T)$ given the new definition of $Z^k_0(T)$, viz.:
\begin{equation}
  Y^k(T) = \{ u \in Z^k_0(T) \ : \ u \perp \rmd Z^{k-1}_0(T) \}.
\end{equation}

We define $G^k(T)$ as follows:
\begin{equation}
  G^k(T) = \{ (f, g) \in \Poly^{r-1} \Lambda^{k}(T) \times \Poly^{r-1} \Lambda^{k-1}(T) \ : \ \rmd^\star f = 0,\ \rmd^\star g = 0\}.
\end{equation}
A motivating observation for this definition is that for $u \in \Poly^{r} \Lambda^{k}(T)$ we have:
\begin{equation}
  (\rmd^\star \rmd u, \rmd^\star u) \in G^k(T).
\end{equation}
Concerning $k = d$ we note that $\rmd^\star f = 0$ means that $f \in \Poly^0 \Lambda^d(T) \approx \bbR$.

We define $B^k(T)$ to be the subspace of $X^k(T)$ consisting of fields $u \in X^k(T)$ such that the following system of PDEs holds, for some $(f,g) \in G^k(T)$:
\begin{align}\label{eq:defApde}
  \int_T \rmd u \cdot \rmd u' + \int_T \rmd u' \wedge a + \int_T u' \wedge b + \int_T u' \cdot \rmd v & = \int_T f \cdot u',\\
  \int_T \rmd u \wedge a' =0 ,\ \int_T u \wedge b' & = 0 ,\\
  \int_T u \cdot \rmd v' & = \int_T g \cdot v',
\end{align}
for all $u' \in X^k_0(T)$, for some $a \in U^{d-k-1}(T)$, some $b \in U^{d-k}(T)$, and some $v \in Y^{k-1}(T)$,  for all $v' \in Y^{k-1}(T)$, all $a' \in U^{d-k-1}(T)$ and all $b' \in U^{d-k}(T)$.

\begin{remark}[Old FES interpretation of VEM and DDR]\label{rem:oldinterpret}
Removing the finite dimensional space of constraints represented by $a,a'$ and $b,b'$, and allowing general $v,v'$, we obtain the harmonic extension used in \cite[\S 2.1]{ChrGil16}. Indeed the equations considered there were:
\begin{align}
  \rmd^\star \rmd u + \rmd v & = f,\\
  \rmd^\star u & = g.
\end{align}
for some $(f,g) \in G^k(T)$. When $k = d$ the top row is replaced by $\int_T u = \int_T f$.

There, we proved that the spaces of solutions to these PDEs define a FES which is \emph{compatible} in the sense of \cite[Definition 5.12]{ChrMunOwr11} and that the spaces $\Pminus^r \Lambda^k(T)$ give unisolvent degrees of freedom. 

From this point of view, the problem was to show that the discrete $\rmL^2$ product, as defined in \cite{DiPDroRap20,BonDiPDroHu25}, is \emph{consistent} with the true $\rmL^2$ product, on this FES, in the sense of Strang's lemmas on variational crimes. We could not settle this question. This paper provides a modification of the PDEs leading to an alternative approach: the PDE is more elaborate but consistency is easy to prove, due to Proposition \ref{prop:hmom} below.
\end{remark}

\begin{remark}[Use of refinements]\label{rem:refine}
In \cite[\S 3]{ChrGil16} (extending \cite{Chr08M3AS} to high order) it is also shown that if a simplicial refinement of the polytopal mesh is introduced, one can define{, by discrete variational techniques,} minimal conforming complexes on the polytopes, consisting of composite finite elements with respect to the refinement, containing given polynomials and with which one can compute. The preference expressed in the VEM community was that one should avoid such refinements and rather learn to compute with virtual functions defined by PDEs. On the other hand we notice that, more recently, DDR has been analysed using other ''lifting'' operators, from the degrees of freedom to standard finite element functions defined on simplicial refinements \cite{DiPDroPit25}. This strategy is similar to the one in the first examples of FES, but the tactics differ.
\end{remark}

For the case $k =0$ of the definition of $B^k(T)$ we have that the system becomes:
\begin{align}
  \int_T \rmd u \cdot \rmd u' + \int_T \rmd u' \wedge a & = \int_T f \cdot u',\\
  \int_T \rmd u \wedge a' & =0 .
\end{align}
For some $a \in U^1(T)$ and all $a' \in U^1(T)$.

Concerning the case $k = d$ we have:
\begin{proposition}
  We have:
  \begin{equation}
    B^d(T) = \Poly^r\Lambda^d(T).
  \end{equation}
\end{proposition}
\begin{proof}
The system becomes:
 \begin{align}
   \int_T u' \wedge b + \int_T u' \cdot \rmd v & = \int_T f \cdot u', \label{eq:firstbd}\\
  \int_T u \wedge b' & = 0 , \label{eq:secondbd}\\
  \int_T u \cdot \rmd v' & = \int_T g \cdot v', \label{eq:lastbd}
\end{align} 
 for all $u' \in X^d(T)$, some $b \in U^0(T)$ and all $b' \in U^{0}(T)$ , some $v \in Y^{d-1}(T)$ and all $v' \in Y^{d}(T)$.

Let $\overline u$ be the $\rmL^2$ projection of $u$ on $\Poly^r\Lambda^d(T)$. Due to the second equation (\ref{eq:secondbd}) we have:
 \begin{equation}
   \int_T (u - \overline u) \wedge b' = 0,
 \end{equation}
 for all $b' \in \Poly^{r+1}(T)$. We can choose $t \in X^{d-1}_0(T)$ such that $\rmd t = u- \overline u$ and for all $c' \in \Pminus^{r+1} \Lambda^1(T)$:
 \begin{equation}
   \int_T t \wedge c' =0.
 \end{equation}
 In particular $t \in Z^{d-1}_0(T)$.

 Since $\rmd^\star g = 0$, the last equation (\ref{eq:lastbd}) is actually true for any $v' \in Z^{d-1}_0(T)$.  Using it for $v' = t$ we get:
 \begin{equation}
   \int_T | \rmd t|^2 = \int_T g \cdot t = \pm \int_T t \wedge \hs g = 0.
 \end{equation}
 We get $u = \overline u$, so $u \in \Poly^r\Lambda^d(T)$.

 This is also sufficient for the existence of a corresponding $(f,g) \in G^d(T)$, since one can choose $f= 0$ and $g = \rmd^\star u$, and remark that $u$ solves the system with $b= 0$ and $v = 0$.
\end{proof}

 \begin{proposition}
   For $k= d$ the system (\ref{eq:defApde}) can be solved iff $f = 0$. In that case $u$ is determined up to a constant.
 \end{proposition}
 \begin{proof}
The first equation can be rewritten:
 \begin{equation}
   \pm \hs b + \rmd v = f.
 \end{equation}
 Integrating this we get $\int_T f =0$. Since $f$ is constant (by the condition $\rmd^\star f =0$), we get $f = 0$.

 On the other hand, if $f =0$, we can find $u \in \Poly^r \Lambda^d(T)$ such that $\rmd^\star u = g$. Then $u$ solves the system with $b= 0$ and $v = 0$.

 We notice that the constants solve the homogeneous problem. Suppose now that $u$ solves the homogeneous problem and $\int_T u  =0$. Since $u \in Z^d(T)$ from the second equation, we have $u \in \rmd Z^{d-1}_0(T) = \rmd Y^{d-1}(T)$. Using $v' \in Y^{d-1}(T)$ such that $u = \rmd v'$ in the last equation gives $u= 0$.
 \end{proof}
 
\begin{proposition}\label{prop:pdeext}
For $k < d$  the above PDE system defining $B^k(T)$, is wellposed and has a unique solution $u$, given boundary data for $u$ in $\partial X^k(T)$ and any right sides $f\in \rmL^2 \Lambda^k(T)$ and $g \in \rmL^2 \Lambda^{k-1}(T)$ with $\rmd^\star f = 0$ and $\rmd^\star g = 0$.
\end{proposition}
\begin{proof}
  As before, in the proof of Proposition \ref{prop:bwellposed}, by essentially identical abstract arguments involving compactness, the inf-sup condition and the coercivity on the kernel hold, for this saddlepoint problem.
\end{proof}

We define $A^k(T)$ as before from $B^k(T)$ viz.:
\begin{equation}
  A^k(T) = \{u \in X^k(T) \ : \ \forall T' \subseteq T \quad u|_{T'} \in B^k(T') \}.
\end{equation}
It is understood that $T' \subseteq T$ means that $T' \in \calT$ is a face of $T$.

We proceed to show that the spaces $A^k(T)$ constitute a FES with some very good properties.

\begin{proposition}[FES]\label{prop:highclosed}
The family of spaces $A^k(T)$ is closed under traces and the exterior derivative.
\end{proposition}
\begin{proof}
  As before, given $u \in B^k(T)$ we prove that $\rmd u \in B^{k+1}(T)$. Given $u'\in Y^k(T)$ we have, since the Lagrange multipliers cancel on $Y^k(T)$:
  \begin{align}
    \int_T \rmd u \cdot \rmd u' = \int_T f \cdot u'.
  \end{align}
  The two finite dimensional constraints required on $\rmd u$ are also satisfied.
\end{proof}

\begin{proposition}[Approximation]
The space $ A^k(T)$ contains $\Poly^r \Lambda^k(T)$.
\end{proposition}
\begin{proof}
  Given $u \in \Poly^r \Lambda^k(T)$ we define $f = \rmd^\star \rmd u$ and $g = \rmd^\star u$.
  Then the PDE system is solved with zero Lagrange multipliers.

  This holds also on the faces of $T$.
\end{proof}

 \begin{proposition}
For $k < d$ the map $G^k(T) \to A^k_0(T)$ obtained from solving (\ref{eq:defApde}) with homogeneous boundary conditions for $u$ is injective.
 \end{proposition}

 \begin{proof}
   Suppose that the solution to (\ref{eq:defApde}) with $u|_{\partial T} = 0$ is $u = 0$. We then get:
    \begin{align}
  \int_T u' \wedge w + \int_T u' \cdot \rmd v & = \int_T f \cdot u',\\
  0 & = \int_T g \cdot v',
\end{align} 
for all $u' \in X^k_0(T)$, for some $w\in W^{d-k}(T)$, all $w'\in W^{d-k}(T)$, some $v \in Y^{k-1}(T)$ and all $v' \in Y^{k-1}(T)$.

With $u' \in Y^k(T)$ we get:
\begin{equation}
  \int_T f \cdot u' =0.
\end{equation}
This also holds for $u' = \rmd v'$ for $v' \in Y^{k-1}(T)$. We deduce that for all $u' \in Z^k_0(T)$:
\begin{equation}
  \int_T f \cdot u' =0.
\end{equation}
By a density argument with respect to $\rmL^2$ norm, this can be extended to any $u' \in \rmL^2 \Lambda^k(T)$ such that:
\begin{equation}
\forall w' \in W^{d-k}(T) \quad  \int_T u' \wedge w' = 0, \textrm{ i.e. }\ u' \perp \hs W^{d-k}(T).
\end{equation}
This implies that $f \in \hs W^{d-k}(T)$.  But in view of (\ref{eq:pmpmw}) we have $\hs W^{d-k}(T) \cap \Poly^{r-1} \Lambda^k(T) = 0$, so we get $f =0$. 

The proof that $g = 0$ is identical.
 \end{proof}

\begin{proposition}[Degrees of freedom]\label{prop:hdof}
  Any element $u \in A^k(T)$ is uniquely determined by the following degrees of freedom, where $T'$ are faces of $T$:
\begin{equation}
  \int_{T'} u \wedge e, \quad e \in \Pminus^r \Lambda^{d'-k}(T'),\ d' = \dim T'.
\end{equation}

\end{proposition}

\begin{proof}
  The case $k = d$ is immediate. We proceed with the case $k <d$. Recall the notation that $A^k_0(T) = A^k(T) \cap X^k_0(T)$. We prove  that the following bilinear form on $A^k_0(T) \times \Pminus^r \Lambda^{d-k}(T)$ is invertible, in several steps:
  \begin{equation}\label{eq:dofa}
    (u, e) \mapsto \int_T u \wedge e.
  \end{equation}

\tpoint  For $u \in A^k_0(T)$ we obtain $u \in Z^k_0(T)$ from the finite dimensional constraints, and, for all $u' \in Z^k_0(T)$, and all $v' \in Y^{k-1}(T)$:
  \begin{align}
    \int_T \rmd u \cdot \rmd u' + \int_T u' \cdot \rmd v & = \int_T f \cdot u',\\
    \int_T  u \cdot \rmd v' & = \int_T g \cdot v'.\label{eq:ytest} 
    \end{align}
  We use $u' = \rmd v$ and use $\rmd^\star f =0$ to deduce $\rmd v = 0$ so $v= 0$.

 \tpoint  Suppose now that for all $e \in \Pminus^r \Lambda^{d-k}(T)$:
  \begin{equation}
    \int_T u \wedge e = 0.
  \end{equation}
Using (\ref{eq:pmpmw}) we then get, actually, that for all $e \in \Pminus^{r+1} \Lambda^{d-k}(T)$:
\begin{equation}\label{eq:uteste}
    \int_T u \wedge e = 0.
  \end{equation}
In particular:
\begin{equation}
  \int_T f \cdot u = 0.
\end{equation}
From the choice $u' = u$ in the first equation, this gives $\rmd u = 0$. For the case $k= 0$ we then deduce $u = 0$.

\tpoint For $0< k <d$ we proceed as follows.

If, in addition $ k> 1$, we remark that for any $v' \in Z^{k-1}_0(T)$ such that $\rmd v' = 0$ there is $v'' \in Z^{k-2}_0(T)$ such that $v' = \rmd v''$. For such $v'$ we have:
\begin{align}
  \int_T u \cdot \rmd v' & = 0,\\
  \int_T g \cdot v' & = \int_T g \cdot \rmd v'' = 0.
\end{align}
Combining this with (\ref{eq:ytest}) we conclude that for all $v' \in Z^{k-1}_0(T)$:
\begin{equation}\label{eq:ztest}
\int_T  u \cdot \rmd v'  = \int_T g \cdot v'. 
\end{equation}
If $ k =1 $ there is no $v' \in Z^{k-1}_0(T)$ such that $\rmd v' =0$, so the above conclusion also holds.

We introduce $\tilde Z^k_0(T)$ defined by:
\begin{equation}
  \tilde Z^k_0(T) = \{ u \in X^k_0(T) \ : \ \forall e \in \Pminus^{r+1} \Lambda^{d-k}(T) \quad \int_T u \wedge e = 0\}.
\end{equation}
We have $\tilde Z^k_0(T) \subseteq Z^k_0(T)$. The complex $\tilde Z^\bs_0(T)$ is exact, by the snake lemma and the exactness properties of $\Pminus^{r+1} \Lambda^{\bs}(T)$. It is a variant of Proposition \ref{prop:zexact} but now exactness of $\tilde Z^\bs_0(T)$ holds also at index $k= d$. 

Since $\rmd u = 0$ and (\ref{eq:uteste}) holds we deduce that there is $t \in \tilde Z^{k-1}_0(T)$ such that $\rmd t = u$. This entails that for all $e \in \Pminus^{r+1} \Lambda^{d-k+1}(T)$:
\begin{equation}
    \int_T t \wedge e = 0.
  \end{equation}
Then we have, setting $u = \rmd t$ and $v' = t$ in (\ref{eq:ztest}):
\begin{equation}
\int_T  \rmd t \cdot \rmd t  = \int_T g \cdot t = \int_T t \wedge \hs g = 0. 
\end{equation}
This gives $u = 0$.

We have proved that the bilinear form (\ref{eq:dofa}) defines an injective map:
\begin{equation}
  A^k_0(T) \to  \Pminus^r \Lambda^{d-k}(T)^\star.
\end{equation}

\tpoint We proceed to prove that the two spaces above have the same dimension.

For $k < d$, we have from \cite[Equation (2.28)]{ChrGil16}:
\begin{equation}
  \dim G^k(T) = \dim \Pminus^r \Lambda^{d-k}(T).
  \end{equation}
Recall also the injectivity of the map $G^k(T) \to A^k_0(T)$, proved above. We get:
\begin{equation}
\dim G^k(T) = \dim A^k_0(T) = \dim  \Pminus^r \Lambda^{d-k}(T).
\end{equation}

For $k = d$ we have more simply:
\begin{equation}
  \dim A^d(T) = \dim \Poly^r \Lambda^d(T) = \dim \Pminus^r \Lambda^0(T).
\end{equation}

It follows that the bilinear form (\ref{eq:dofa}) is invertible.

\tpoint Now  it follows from Proposition \ref{prop:pdeext} that the FES \emph{admits extensions} in the sense of \cite[Definition 2.3]{ChrRap16} (or \cite[Definition p. 75]{ChrMunOwr11}). The unisolvence of the degrees of freedom follows then from a general induction argument given for FES in \cite[Proposition 2.5]{ChrRap16} (or \cite[Proposition 5.35]{ChrMunOwr11}).
\end{proof}

\begin{proposition}[Computation of moments]\label{prop:hmom}
For elements $u \in A^k(T)$ its moments against $\Poly^r\Lambda^{d-k}(T)$ can be computed inductively from the degrees of freedom.
\end{proposition}

\begin{proof}
\tpoint Let $u \in A^k(T)$.  Suppose that $w \in \Poly^r\Lambda^{d-k}(T)$ and that we want to compute:
  \begin{equation}
    \int_T u \wedge w.
  \end{equation}
  We write as before:
  \begin{equation}
    (d- k +r) w = \rmd \kappa_T w + \kappa_T \rmd w - Q_{r-1} w.
  \end{equation}
  On the right side the two last terms are degrees of freedom, so we concentrate on the first. We have:
  \begin{equation}\label{eq:uwk}
    \int_T u \wedge \rmd \kappa_T w = \pm \int_{\partial T} u \wedge \kappa_T w \pm \int_T \rmd u \wedge \kappa_T w.
  \end{equation}
  
  \tpoint As an intermediate step, we now treat the case where $\rmd u = 0$. Then there is only the first term on the right side. It reduces to terms of the following form, where $T'$ is a codimension one face of $T$:
  \begin{equation}\label{eq:comput}
    \int_{T'} u \wedge (\kappa_T - \kappa_{T'}) w + \int_{T'} u \wedge (\kappa_{T'} w - P_r \kappa_{T'} w) +  \int_{T'} u \wedge P_r \kappa_{T'} w.
  \end{equation}
  Here $P_r$ is $L^2$ projection onto polynomial differential forms of polynomial degree $r$. Since $P_r$ is applied to polynomial forms (or degree at most $r+1$) on a polytope, it is \emph{computable}.
  
  In (\ref{eq:comput}) the first and last terms are moments against $\Poly^r\Lambda^{d-1-k}(T')$, so amena\-ble to an induction step. The middle term is zero, since it is a moment against $U^{d-1}(T')$ representing a constraint.

  This settles the intermediate step: we can compute the moments of $u$ against polynomials of degree $r$ in the case $\rmd u = 0$, from the dofs of $u$, by induction on dimension.

  \tpoint In the general case we remark that $\rmd \rmd u = 0$ and that the degrees of freedom of $\rmd u$ are computable from those of $u$. Applying the intermediate step to $\rmd u$, its moments against polynomials of degree $r$ can be computed. Going back to (\ref{eq:uwk}) and looking at the last term, we write:
  \begin{equation}
    \int_T \rmd u \wedge \kappa_T w = \int_T \rmd u \wedge (\kappa_T w - P_r \kappa_T w) + \int_T \rmd u \wedge P_r \kappa_T w.
  \end{equation}
  On the right side the first term is zero by construction, since  $\kappa_T w - P_r \kappa_T w\in U^{d-k-1}$. The last term is computed by the above intermediate step.

  In (\ref{eq:uwk}) we are left with the first term on the right side, for which (\ref{eq:comput}) still applies. Again, this provides the beginnings of an induction procedure. 
\end{proof}

\begin{remark}[Exterior derivative and degrees of freedom]\label{rem:extderdof}
If $u\in A^k(T)$ and we let $v$ be a degree of freedom acting on $\rmd u \in A^{k+1}(T)$, we write:
  \begin{equation}
    \int_T \rmd u \wedge v = \int_{\partial T} u \wedge v - (-1)^k \int_T u \wedge \rmd v.
  \end{equation}
  The right side only involves degrees of freedom acting on $u$. From this identity it follows that the discrete exterior derivative is computable in the chosen degrees of freedom. Another consequence is that the interpolator associated with the degrees of freedom commutes with the exterior derivative (directly or by \cite[Proposition 5.41 p. 86]{ChrMunOwr11}).
\end{remark}

\begin{remark}[The vocabulary of FES] We now relate some of the above results to the vocabulary of the FES framework.

  That the family of spaces $A^k(T)$ is closed under traces and the exterior derivative, as proved in Proposition \ref{prop:highclosed}, is the definition of a FES. However, not all FES produce global spaces that are suitable for Galerkin methods. Two additional conditions are important.

  The first condition, referred to as \emph{softness}, has several equivalent formulations, see especially \cite[Propositions 2.5 and 2.6]{ChrRap16}. One requires that discrete fields on the boundary of a cell, that satisfy a compatibility condition with respect to faces inside the boundary, can be extended to a discrete field on the cell. Here, this follows from Proposition \ref{prop:pdeext}. This condition turns out to be equivalent to the existence of unisolvent degrees of freedom which we have also proved, in Proposition \ref{prop:hdof}.

  The second condition, referred to as local exactness, is that for each cell $T$, the sequence $A^\bs(T)$ resolves $\bbR$. Here also, many stategies are possible. Given that softness holds, one can prove instead that the sequences $A^\bs_0(T)$ are exact, except at index $\dim T$ where the cohomology is determined by the integral, see \cite[Proposition 2.3]{ChrRap16}. This seems easier in this particular case. But even easier is to deduce the sequence exactness from a property of the degrees of freedom as in \cite[Proposition 2.7 and 2.8]{ChrRap16}, based on Remark \ref{rem:extderdof} above.

  The two algebraic conditions, softness and local exactness, ensure that we have a \emph{compatible} FES. When a compatible FES contains polynomials of given degree, is stable with respect to small perturbations of the mesh, and is invariant under scalings, it behaves like standard finite element spaces, such as the spaces $\Pminus^r \Lambda^k(T)$ on simplices{, in terms, for instance, of convergence properties of associated algorithms. In general this principle relies on Remark \ref{rem:fesperturb}.}
\end{remark}

\section{Some schemes and their analysis\label{sec:algo}}
In this section we sketch some numerical methods and elements of their analysis, relative to a choice of polynomial degree $r$. The choice $r=0$ is the lowest order case described in \S \ref{sec:low}, whereas the case $r\geq 1$ was described in \S \ref{sec:high}. On the domain $S$ in $\bbR^n$, a shape-regular sequence of polytopal meshes $\calT_h$ is considered (in the sense \cite[Assumption 2.1]{BeiBreMarRus23}).

A complete convergence analysis depends on the Hypothesis of Stability of Constants (HSC), as expressed in Remark \ref{rem:fesperturb}. In the present setting it is conjectural. We explain how it enters the arguments. The results proved in the previous sections were mostly obtained for individual cells. It is in passing to results that are uniform over all the cells, in all the meshes, that HSC enters.

By $a \cleq b$ we mean $a \leq C b$ for some constant $C$ independent of the quantities considered. Dependence of $C$ on the cell $T$ and the scale $h$ are usually discussed explicitly.

\revised{
Referring to the discussions in \cite{LioMag72,KonOle83,MazRos10,CosDauNic12} we introduce weighted Sobolev seminorms. It allows us to recover variants of elliptic regularity results on non-smooth domains like the cells. 
In a cell $T$ of diameter 1, let $\rho$ be the distance to the boundary $\partial T$. For $u \in A^k(T)$ we have, for any $q \in \bbN$:
\begin{equation}
  | \rho^q \nabla^q u | < \infty.
\end{equation}
Here $|\cdot|$ is the usual $\rmL^2(T)$ norm. For $q=0$ this expression is taken to mean the usual $\rmL^2$ norm.

The sharpness of this result will not matter. We could replace $\rho^q$ by any function that goes at least as fast to $0$ near the boundary, and still conclude. By the analyticity of elements of $A^k(T)$ on the interior of $T$ (hypoellipticity of the system of PDEs),  we could even use for $\rho$ the characteristic function of a ball with closure in the interior of $T$. 

We scale this seminorm as follows. When $T$ is a cell of diameter $h$, let $\hat T$ be the scaled cell of diameter $1$ equipped with the weight $\hat \rho$ defined previously. The weight $\rho$ on $T$ is given by $\rho: x \mapsto \hat \rho(x/h)$. Thus, on $T$, when $\hat \rho$ is the distance to the boundary in $\hat T$, $\rho$ is $1/h$ times the distance to the boundary in $T$. The weight is pulled back as a function (not as a volume form for instance). The seminorm then scales with the same power of $h$ as without the weight.

}

\subsection{Discrete local $\rmL^2$ product} Given $u \in A^k(T)$ and based on the computation of moments against forms which are polynomials of degree $r$ obtained in Propositions \ref{prop:momconst} and \ref{prop:hmom}, one is able to inductively compute $\gamma_T u$, which is the $\rmL^2(T)$ projection of $u$ on $\Poly^r \Lambda^k (T)$.

Following VEM practice \cite{BeiBreMarRus23} one also introduces a symmetric bilinear form $s_T$ on $A^k(T)$, defined from the degrees of freedom, which is bounded and coercive in $\rmL^2$, uniformly with respect to $T$. It should scale like the $\rmL^2$ inner product (under scaling maps $x \to h_T x$, where $h_T$ is the diameter of $T$) but no consistency is required. It can be block diagonal in the degrees of freedom. A natural choice for $s_T$ is:
\begin{equation}
  s_T(u, u') = \sum_{T' \subseteq T} h_{T'}^{\dim T - \dim T'} \int_{T'} \gamma_{T'} u \cdot \gamma_{T'} u'.
\end{equation}
The sum is over faces $T'$ of $T$. The scaling factor mimicks that of the $\rmL^2$ product on $T$. In this case $\gamma_T$ is block triangular in the degrees of freedom.

One then defines the discrete $\rmL^2$-product on $A^k(T)$ by:
\begin{equation}\label{eq:l2h}
  \langle u, u' \rangle_T = \int_T \gamma_T u \cdot \gamma_T u' + s_T( u - \gamma_T u, u' -\gamma_T u').
\end{equation}
The second term on the right side is referred to as the stabilizing term.
The crucial points are that this discrete inner product is computable, coercive and consistent:

\begin{itemize}
\item \emph{Computability} is clear from the computability of $\gamma_T$, from the degrees of freedom in $\Pminus^r \Lambda^k(T')$ ($T' \subseteq T$) to $\Poly^r \Lambda^k(T)$.

\item \emph{Coercivity.} We justify it first for the stabilizing term $s_T$.  We remark that, by Propositions \ref{prop:ldof} and \ref{prop:hdof}, for any fixed $T$, the map:
  \begin{equation}
    \mapping{A^k(T)}{\prod_{T' \subseteq T} \Poly^r\Lambda^k(T')}
      {u}{(\gamma_{T'} u)_{T'}}
  \end{equation}
  is injective. This gives:
  \begin{equation}
    s_T(u,u) \approx |u|^2_{\rmL^2(T)}.
  \end{equation}
  Under HSC this equivalence of norms holds uniformly for all relevant cells $T$ of diameter $1$. Then by Euclidean motions and scalar scalings $x \mapsto h_T x$ it holds for all the cells in all the meshes considered.

Next we justify the coercivity for the discrete $\rmL^2$ product defined in (\ref{eq:l2h}), compared with the standard $\rmL^2$ product, denoted $|\cdot |$. For  $u \in A^k(T)$:
\begin{equation}
  \langle u, u \rangle_T \approx |\gamma_T u|^2 + |u - \gamma_T u|^2 = |u|^2,
\end{equation}
based on the coercivity of $s_T$.

\item \emph{Consistency} comes in an algebraic and an analytic form. Since $\gamma_T$ is the (true!) $\rmL^2$ projection onto $\Poly^r\Lambda^k(T)$, the discrete inner product is algebraically consistent in the sense that if at least one of $u, u'$ is in $\Poly^r \Lambda^k(T)$, it gives the exact value:
\begin{equation}\label{eq:l2exact}
  \langle u , u' \rangle_T = \int_T u \cdot u'.
\end{equation}
For a fixed cell $T$ it follows that we have an estimate, for $u,u' \in A^k(T)$, and any $0 \leq q,q' \leq r+1$:
\begin{equation}
  | \langle u, u' \rangle_T - \int_T u \cdot u' | \cleq |\weight \nabla^{q} u| \, | \weightp \nabla^{q'} u'|.
\end{equation}
Under HSC this inequality holds uniformly for relevant cells of diameter $1$.
Then by Euclidean motions and scalar scalings $x \mapsto h_T x$, it holds, for all the cells in all the meshes considered, that:
\begin{equation}\label{eq:qcon}
  | \langle u, u' \rangle_T - \int_T u \cdot u' | \cleq h_T^{q+q'} |\weight \nabla^{q} u| \, | \weightp \nabla^{q'} u'|.
\end{equation}
This is the quantitative analytic version of consistency.
\end{itemize}

\subsection{Global constructs} For any given mesh $\calT$ the terms of individual maximal cells $T \in \calT$ are summed up:
\begin{equation}
  \langle u, u' \rangle_h = \sum_{T} \langle u,u' \rangle_T,
\end{equation}
Both coervivity and consistency pass to a global form, on the whole domain $S$.

Recall that the exterior derivative $\rmd: A^k(T)\to A^{k+1}(T)$ is computable from the degrees of freedom on $A^k(T)$ to those on $A^{k+1}(T)$, as pointed out in Remark \ref{rem:extderdof}.

From there, numerical methods for various PDEs can be defined. Property (\ref{eq:l2exact}) gives (\ref{eq:qcon}) which make these methods amenable to an analysis based on Strang's lemmas. Such studies have been carried out in the context of VEM, as illustrated with a complete model problem in \cite[\S 3]{BeiBreMarRus23}.

\begin{remark}[Comparison with DDR]\label{rem:compare}
  Summing up, discrete methods hinge on two definitions: the discrete exterior derivative and the discrete inner product.
  \begin{itemize}
  \item Here, the discrete exterior derivative is the continuous one, and is computed according to Remark \ref{rem:extderdof}. This situation should be compared with Equations (3.4) (local operator) and (3.8) (global operator) in \cite{BonDiPDroHu25}. The ingredients are similar but our definition does not involve the ''discrete potentials'' (i.e. their $P^k_{r,f}$ which are denoted $\gamma_T$ here).
  \item Here, the discrete $\rmL^2$ inner product is defined in Equation (\ref{eq:l2h}). In \cite{BonDiPDroHu25}, the definition is given in \S 3.1.5. It also has two terms. The first (consistent) term is defined in \S 3.1.3 and should be compared with our definition of $\gamma_T$ as computed by Proposition \ref{prop:hmom}. If the ingredients are the same, the presentation is different. The second term (which stabilises without ruining consistency) is the same in our and their presentations.
  \end{itemize}
    \removed{Even if an exact correspondence has not been established, we regard possible differences between the algorithms to be matters of minor variations with similar ingredients. In any case DDR was the inspiration for our arguments.} \revised{An exact correspondence between FES and DDR is provided in \S \ref{sec:cor}.}
\end{remark}

As an example of PDE on differential forms that features many generic difficulties (such as the need for inf-sup conditions in the analysis), we consider the Hodge-Laplace problem \cite{ArnFalWin10}. Given a $k$-form $f$ on $S$ one seeks a $k$-form $u$ on $S$ such that:
\begin{equation}
  - \Delta u = f, \ \textrm{with } -\Delta = \rmd^\star \rmd + \rmd\, \rmd^\star,
\end{equation}
with, say, natural boundary conditions on $\partial S$. We suppose the domain is contractible for simplicity. One introduces the auxiliary variable $p = \rmd^\star u$ which is a $(k-1)$-form. We let $X^k$ and $X^{k-1}$ be the global spaces of $k$- and $(k-1)$-forms, which are $\rmL^2$ with exterior derivative in $\rmL^2$. The mixed formulation in terms of $(u, p) \in X^k \times X^{k-1}$ is:
\begin{equation}
\mixedright{u^\test}{X^k}{\int \rmd u \cdot \rmd u^\test + \int u^\test \cdot \rmd p}{l(u^\test) = \int f \cdot u^\test, }{p^\test}{X^{k-1}}{ \int u \cdot \rmd p^\test - \int p \cdot  p^\test}{0.}
\end{equation}
For essential rather than natural boundary conditions, one just replaces $X^k \times X^{k-1}$ by $X^k_0 \times X^{k-1}_0$ in this system.
We discretize the above problem as, find $(u,p) \in X_h^k \times X_h^{k-1}$ such that:
\begin{equation}\label{eq:mixedah}
\mixedright{u^\test}{X^k_h}{\langle \rmd u , \rmd u^\test \rangle_h + \langle u^\test, \rmd p \rangle_h}{l_h(u^\test) \sim l(u^\test), }{p^\test}{X^{k-1}_h}{ \langle u, \rmd p^\test \rangle_h-  \langle p ,  p^\test \rangle_h }{0.}
\end{equation}
This problem is of the type studied in \cite[\S 4.3]{BofBreFor13}.

\subsection{Smoothed projections}
For the convergence of such methods one establishes some more stability and consistency estimates. We use the method of smoothed projections, following \cite[\S 5.3]{ChrMunOwr11}. It has been illustrated also in \cite{Chr18}. See also \cite{ChrHal11SINUM,ChrHal15} for some techniques involved in handling consistency.
  
We let $a$ and $a_h$ be the bilinear forms on $X^k \times X^{k-1}$ and $X_h^k \times X_h^{k-1}$ describing these mixed formulations, viz:
\begin{equation}
    a_h(u,p \,; u^\test,p^\test) = \langle \rmd u , \rmd u^\test \rangle_h + \langle u^\test, \rmd p \rangle_h + \langle u, \rmd p^\test \rangle_h-  \langle p ,  p^\test \rangle_h.
\end{equation}
The above consistency result on the $\rmL^2$ products gives, for any $0 \leq q \leq r+1$:
\begin{equation}
    |a(u,p \,; u^\test,p^\test)- a_h(u,p \,; u^\test,p^\test)| \cleq h^q (\lfloor u \rfloor_{h,q} + \lfloor p \rfloor_{h,q})(\|u^\test\| + \|p^\test\|), \label{eq:acon}
\end{equation}
where, for instance, the broken semi-norm $\lfloor \cdot \rfloor_{h,q}$ applied to $u$ incorporates $q$ derivatives of $u$ and $\rmd u$ on each maximal cell $T \in \calT$, viz.:
\begin{equation}
  \lfloor u \rfloor_{h,q}^2 = \sum_{T} | \weight \nabla^q u|_{\rmL^2(T)}^2  + |\weight \nabla^q \rmd u|_{\rmL^2(T)}^2.
\end{equation}
Recall also the notation $(\ref{eq:gnorm})$ for the norm on $X^k$.
  
We let $I_T$ be the interpolator onto $A^k(T)$ associated with the chosen degrees of freedom.
On a given cell $T$ we have for $1\leq q \leq r+1$, for all $u$ in $\rmH^q(T)$ with all traces to faces $T'$ of $T$ in $\rmH^q(T')$:
\begin{align}
  |u - I_T u|_{\rmL^2(T)}  & \cleq \sum_{T' \subseteq T} |\nabla^q u|_{\rmL^2(T')},\\
  | \weight \nabla^q I_T u|_{\rmL^2(T)} & \cleq \sum_{T' \subseteq T} |\nabla^q u|_{\rmL^2(T')}.
\end{align}
Under HSC these inequalities hold uniformly for all relevant $T$ of size $1$. Then scaling gives:
\begin{align}
  |u - I_T u|_{\rmL^2(T)} & \cleq \sum_{T' \subseteq T} h_T^{q + (\dim T - \dim T')/2}|\nabla^q u|_{\rmL^2(T')},\\
  | \weight \nabla^q I_T u|_{\rmL^2(T)} & \cleq \sum_{T' \subseteq T}  h_T^{(\dim T - \dim T')/2} |\nabla^q u|_{\rmL^2(T')}.
\end{align}
Compared with standard estimates based on the Bramble-Hilbert or Deny-Lions lemmas, such as \cite[Theorem 16.1 p. 126]{Cia91}, extra terms on faces are incorporated on the right side. Given that $q \geq 1$ they help make the interpolator well defined. They are controlled by a smoothing step.

The prototypical smoother is defined by convolution.
Let $\phi$ be a smooth function on $\bbR^n$ with support in the open unit ball, with integral $1$ and such that convolution by $\phi$ preserves polynomials of degree up to $r$. Then we define, for $\delta >0$, $x \in \bbR^n$:
\begin{equation}
  \phi_\delta (x) = 1 / \delta^n \phi(x/\delta).
\end{equation}
The smoother $R^\epsilon_h$, which is adapted to $\calT_h$ is then defined by:
\begin{equation}
  R^\epsilon_h  u = \phi_{\epsilon h} \ast u.
\end{equation}
As in \cite{Chr07NM}, the parameter $\epsilon$ is fixed but small enough.  In particular the parameter $\epsilon$ ensures that the $\epsilon h_T$ -neighborhood of $T$ is inside the macroelement consisting of elements immediately surrounding $T$ in  $\calT_h$. The smoothed interpolator $I_h R^\epsilon_h$ yields quite sharp results for periodic boundary conditions and quasi-uniform meshes, as illustrated in \cite{Chr18}.

The technique was extended to other boundary conditions and non quasi-uniform meshes in \cite{ChrWin08} by generalizing the definition of $R^\epsilon_h$ (with now $h$ a smooth function varying over $S$ and locally equivalent to $h_T$ near $T \in \calT_h$). It is essential, for the estimates, that $R^\epsilon_h$ commutes with the exterior derivative, just like $I_h$. Polynomial preservation was achieved and leveraged in \cite{ChrMunOwr11}.

We conclude that we may define a global smoother $R^\epsilon_h$ such that, locally, we have the estimates for any $0\leq q' \leq r+1$ and $T' \subseteq T$:
\begin{equation}
   h_T^{ q+ (\dim T - \dim T')/2} |\nabla^q R^\epsilon_h u|_{\rmL^2(T')} \cleq  h_T^{ q'} | \nabla^{q'} u|_{\calV_h^\epsilon (T)}. 
\end{equation}
On the right $\calV^\epsilon_h(T)$ is the $\epsilon h_T$-neighborhood of $T$ in $S$. We then have a compensation effect in the scalings on lower dimensional cells in the above estimates:
\begin{align}
  |u - I_T R^\epsilon_h u|_{\rmL^2(T)} & \cleq h_T^{q} | \nabla^{q} u|_{\rmL^2(\calV_h^\epsilon (T))},\label{eq:approxir}\\
  | \weight \nabla^q I_T R^\epsilon_h u |_{\rmL^2(T)} & \cleq | \nabla^q u |_{\rmL^2(\calV_h^\epsilon (T))}.\label{eq:stabnabla}
\end{align}
These estimates can be globalized, by summation of the squares, taking into account that on the right side the overlaps of macroelements are controlled.

To deduce discrete inf-sup conditions from continuous inf-sup conditions it is useful to have commuting \emph{projections}.  For any $\delta > 0$ we may find $\epsilon_0 >0 $ so that for $\epsilon < \epsilon_0$, we have, for all $u \in X^k_h$:
\begin{equation}
  |u - I_h R^\epsilon_h u|  \leq \delta  |u| .
\end{equation}
This is done locally, using HSC, and then globalized. As an operator $I_h R^\epsilon_h : X^k_h \to X^k_h$ we can get $ \| \id - I_h R^\epsilon_h \| \leq \delta$ in operator norm, for any choice $\delta <1$, hence it may be inverted on $X_h$. We get uniformly $\rmL^2$ bounded commuting projections:
\begin{equation}
\mapping{\rmL^2 \Lambda^k(S)}{X^k_h,}  
  {u}{ (I_h R^\epsilon_h|_{X^k_h})^{-1} I_h R^\epsilon_h u.}
\end{equation}
This is known as the Schöberl trick, introduced in connection with \cite{Sch08}. Sometimes it is useful to have a projection, but below we illustrate that it can also sometimes be preferable to work directly with $I_h R^\epsilon_h$.

\subsection{Convergence} We define:
\begin{align}
  W_h & = \{u \in X^k_h \ : \ \rmd u = 0\},\\
  V_h & = \{u \in X^k_h \ : \ \forall w \in W_h \quad \langle u, w \rangle = 0\},\\
  \widetilde V_h & = \{u \in X^k_h \ : \ \forall w \in W_h \quad \langle u, w \rangle_h = 0\}.
\end{align}
\begin{proposition}[Poincaré-Friedrichs]
  There exists $C > 0$ such that for all $u \in \widetilde V_h$ we have:
  \begin{equation}
    | u|_h \leq C | \rmd u |_h.
  \end{equation}
\end{proposition}
\begin{proof}
  \tpoint  The difficulty is not so much that discrete $\rmL^2$ norms $|\cdot|_h$ appear here, than that it concerns elements of $\widetilde V_h$ rather than $V_h$. For $V_h$ the Poincaré-Friedrichs estimate follows from the existence of uniformly bounded projections, sketched above, as discussed for instance in \cite[Proposition 4.4]{ChrWin13IMA}.

\tpoint Given $u \in \widetilde V_h$, we find $v \in V_h$ such that $u$ is the projection of $v$ onto $\widetilde V_h$ with respect to $\langle \cdot , \cdot \rangle_h$. This can also be expressed by $u-v \in W_h$, i.e. $\rmd u = \rmd v$.
We then have:
\begin{align}
  |u|_h \leq |v|_h \cleq |v| \cleq |\rmd v| \cleq |\rmd v|_h = |\rmd u|_h. 
\end{align}
The third inequality comes from (i) above. The second and fourth come from the above coercivity results.
\end{proof}

\begin{remark}\label{rem:eqV}
  The situation in the above proof is that we have $u \in \widetilde V_h$ and $v \in V_h$, such that $u-v \in W_h$. The space $W_h$ is orthogonal to $V_h$ with respect to $\langle \cdot, \cdot \rangle$ and to  $\widetilde V_h$ with respect to $ \langle \cdot , \cdot \rangle_h$. Thus $v$ is the projection of $u$ onto $V_h$ with respect to $\langle  \cdot, \cdot \rangle$ and $u$ is the projection of $v$ onto $\widetilde V_h$ with espect to $\langle  \cdot, \cdot \rangle_h$. It follows that we have a uniform equivalence of norms:
  \begin{equation}
    | u | \approx | v|. 
  \end{equation}
\end{remark}
This proposition gives coervivity on the kernel for the topleft term in (\ref{eq:mixedah}). The corresponding result in $X^{k-1}_h$ instead of $X^k_h$, gives the inf-sup compatibility of the two off diagonal coupling terms in (\ref{eq:mixedah}). Together they are enough to establish the stability of the above discrete Hodge-Laplace problem in the sense of a uniform discrete inf-sup condition for the bilinear form $a_h$. See \cite[Theorem 4.3.1]{BofBreFor13}.

We denote the variable $(u,p)\in X^k_h \times X^{k-1}_h$ as $v \in X_h$. Suppose that the true solution is a $v$ which is $\rmH^q(S)$ with exterior derivative in $\rmH^q(S)$. The discrete solution is denoted $v_h \in X_h$. Strang's first lemma \cite[Theorem 26.1 p. 192]{Cia91} (extended to the inf-sup situation) then gives the estimate:
\begin{align}
  \| v - v_h \| \cleq & \ \| v - I_h R^\epsilon_h v\| + \sup_{v' \in X_h}\frac{| (a - a_h)( I_h R^\epsilon_h v, v')|}{\| v'\|} \ + \\[1ex]
  & \sup_{v' \in X_h}\frac{|(l-l_h)(v')|}{\|v'\|}.
\end{align}
On the right side we notice that:
\begin{itemize}
\item the first term was estimated in (\ref{eq:approxir}),
\item the second term was estimated in (\ref{eq:stabnabla}) and (\ref{eq:acon}),
\item the third term, which is of an altogether different nature, is left to the reader.
\end{itemize}
All in all, we get convergence of order $q$ of the discrete solution in the norm of $X^k$, as long as the true solution is $\rmH^q$ with exterior derivative in $\rmH^q$, and $q \leq r+1$.

\subsection{Discrete compactness} For corresponding eigenvalue problems such as those of Maxwell \cite{Bof10,ChrWin13IMA} one needs a stronger result called discrete compactness. Some further notations:
\begin{align}
  W & = \{u \in X^k \ : \ \rmd u = 0\},\\
  V & = \{u \in X^k \ : \ \forall w \in W \quad \langle u, w \rangle = 0\}.
\end{align}
Let $P_V$ be the $\rmL^2$ projection onto $V$. We use that we have $\rmL^2$ bounded commuting projections $\Pi_h$ onto $X^k_h$. We start with a $v \in V_h$. We notice that $v$, $P_V v$ and $\Pi_h P_V v$ have the same exterior derivative. So $v - \Pi_h P_V v \in W_h$ is $\rmL^2$ orthogonal to both $v$ and $P_V v$. We deduce: 
\begin{equation}
  |v - P_V v|^2 + |v - \Pi_h P_V v |^2 = |P_V v - \Pi_h P_V v|^2.  
\end{equation}
Hence
\begin{equation}
  |v - P_V v| \leq  |P_V v - \Pi_h P_V v|.
\end{equation}
The space $V$, which is equipped with the norm of $X^k$, encodes also vanishing coderivative. An alternative norm on $V$ is $|\rmd \cdot |$. The injection of $V$ into $\rmL^2$ is compact. One deduces that for some sequence $\epsilon_h \to 0$:
\begin{equation}
  |v - P_V v| \leq \epsilon_h | \rmd v|.
\end{equation}
This gives the vanishing gap condition (concerning the one-sided gap from $V_h$ to $V$ in $X^k$ norm), which is equivalent to discrete compactness \cite[Proposition 3.9]{ChrWin13IMA}. If $(v_h)$ is a sequence in $\widetilde V_h$ with $\rmL^2$ bounded $(\rmd v_h)$, one can get a subsequence converging in the $\rmL^2$ norm, simply by considering $(P_V v_h)$.

We now prove the vanishing gap condition also for $\widetilde V_h$.
\begin{proposition}[Vanishing gap]
  There is a sequence $\epsilon_h \to 0$ such that when $u \in \widetilde V_h$ and $v \in V_h$ is defined (as before) by $\rmd v = \rmd u$ we have:
  \begin{equation}
    \| u - v\| \leq \epsilon_h \| u\|.
  \end{equation}
\end{proposition}

\begin{proof}
  Let $Q_h$ be the map $X^k \to X^k_h$ defined so that, for $t' \in X^k_h$:
  \begin{equation}
    \langle Q_h t,t' \rangle_h = \langle t, t'\rangle. 
  \end{equation}
  It is not a projection, but it maps $V$ to $\widetilde V_h$ and satisfies approximation properties (by stability and consistency).
  
We have:
\begin{align}
  |u - v|^2_h & = \langle u - v, Q_h P_V u - v \rangle_h \leq | u-v |_h |Q_h P_V u - v|_h.
\end{align}
We get, remarking that $P_V u = P_V v$ (they have the same exterior derivative):
\begin{align}
  |u -v| & \cleq |Q_h P_V u - v| \leq |Q_h P_V u - P_V u| + |P_V v - v|,\\
  & \cleq \epsilon_h | \rmd u|,
\end{align}
using the already established results.
\end{proof}

\begin{remark}
  This proposition can be interpreted as saying that the symmetrized gap $\delta(\widetilde V_h, V_h)$, computed in the norm $\| \cdot \|$ of $X^k$, converges to $0$ as $h \to 0$. On the other hand Remark \ref{rem:eqV} can be interpreted as saying that the symmetrized gap $\delta(\widetilde V_h, V_h)$ computed in the $\rmL^2$ norm $| \cdot|$ stays away from 1.
\end{remark}

\revised{
\appendix 

\section{Exact equivalence\label{sec:cor}}
The correspondance between FES and DDR provided in {Remark 5.1} can be made completely precise. We treat the case $r \geq 1$.

A cell $T$ with faces $T'$ is now typically denoted $f$ with faces $f'$. The dimension of $f$ is $d$, that of $f'$ is denoted $d'$.

For $u\in A^k(f)$, its degrees of freedom in the sense of FES, given by {Proposition 4.9}, are the linear forms $\int_{f'} u \wedge (\cdot) \in \Pminus^r \Lambda^{d'-k}(f')^\star$, for each face $f'$ of $f$. In DDR such linear forms are represented by elements $\omega_{f'} \in \Pminus^r \Lambda^{d'-k}(f')$, through $\rmL^2$ duality, and collected in the discrete space $\underline X^k_{r,f}$, which is a simple product  \cite[page 11]{BonDiPDroHu25}:
\begin{equation}
  \omega = (\omega_{f'})_{f' \subseteq f} \in \prod_{f' \subseteq f}  \Pminus^r \Lambda^{d'-k}(f') = \underline X^k_{r,f}.
\end{equation}
More explicitely, we will use that a $u \in A^k(f)$ corresponds bijectively to a family of fields $\omega \in \underline X^k_{r,f}$, characterized by the property that for each face $f'$ of $f$, and all $\mu \in \Pminus^r \Lambda^{d'-k}(f')$ we have:
\begin{align}
  \int_{f'} u \wedge \mu & = \int_{f'} \omega_{f'} \cdot \mu,\\
  & = \int_{f'} \hs^{-1} \omega_{f'} \wedge \mu.
\end{align}
Notice that $\hs^{-1}  \Pminus^r \Lambda^{d'-k}(f')$ is a strict subspace of $\Poly^r \Lambda^k(f')$, except when $k = d'$. In DDR, the collection $\omega \in \underline X^k_{r,f}$ is reconstructed as $P_{r,f} \omega \in \Poly^r \Lambda^k(f)$, referred to as the discrete potential, whereas the discrete exterior derivative is computed as $\rmd_{r,f} \omega \in \Poly^r \Lambda^{k+1}(f)$. These two operators are defined in \cite[\S 3.1.3]{BonDiPDroHu25}, by a coupled induction on dimension.
\begin{proposition}\label{prop:corfesddr}
  We have, for $u \in A^k(f)$ corresponding to $\omega \in \underline X^k_{r,f}$:
  \begin{align}
     \rmd_{r,f} \omega & =  \gamma_f (\rmd u),\\
     P_{r,f}\omega &= \gamma_f u.
  \end{align}
\end{proposition}

\begin{proof}
By a coupled induction on the dimension $d$.
  
\tpoint Case $d=k$.

\noindent The first equality, concerning the exterior derivative, is trivial.
  
\noindent The second equality follows from comparing {Proposition 4.3} and \cite[Eq. (3.3)]{BonDiPDroHu25}.

\tpoint Case $d > k$.

\noindent \emph{Discrete exterior derivative.} We start from the definition in \cite[Eq. (3.4)]{BonDiPDroHu25}. For $\mu \in \Poly^r \Lambda^{d-k-1}(f)$:
\begin{align}
  \int_f \rmd_{r,f} \omega \wedge \mu & =  (-1)^{k+1}  \int_f \hs^{-1} \omega_f \wedge \rmd \mu  +  \sum_{f'} \int_{f'} P_{r,f'} \omega \wedge \mu|_{f'},\\
  & =  (-1)^{k+1}  \int_f \hs^{-1} \omega_f \wedge \rmd \mu  +  \sum_{f'} \int_{f'} \gamma_{f'} u \wedge \mu|_{f'},\\
& =  (-1)^{k+1} \int_f u \wedge \rmd \mu + \int_{\partial f} u \wedge \mu|_{\partial f},\\
& = \int_f \rmd u \wedge \mu =  \int_f \gamma_f (\rmd u) \wedge \mu. 
\end{align}
Summation, here and in the next point, is over codimension 1 faces $f'$ of $f$, on which we use the induction hypothesis.

\noindent \emph{Discrete potential.} The $\rmL^2(f)$ projection onto $\Poly^r \Lambda^k(f)$, is denoted  $\pi_{r,f}$.  We start from the definition \cite[Eq. (3.5)]{BonDiPDroHu25}. For $\mu \in \kappa \Poly^{r} \Lambda^{d-k}(f)$ we have:
 \begin{align}
  (-1)^{k+1} \int_f P_{r,f} \omega \wedge \rmd \mu & = \int_f \rmd_{r,f} \omega \wedge \mu - \sum_{f'} \int_{f'} P_{r,f'} \omega \wedge \mu,\\
   & = \int_f \gamma_f (\rmd u) \wedge \pi_{r,f} \mu  - \sum_{f'}\int_{f'} \gamma_{f'} u \wedge \pi_{r,f'} \mu  ,\\
   & = \int_f \rmd u \wedge \pi_{r,f} \mu  - \sum_{f'}\int_{f'} u \wedge \pi_{r,f'} \mu, \label{eq:beforein}\\
  & = \int_f \rmd u \wedge \mu -  \sum_{f'}\int_{f'} u \wedge \mu,\\ 
  & = (-1)^{k+1} \int_f u \wedge \rmd \mu = (-1)^{k+1} \int_f \gamma_f u \wedge \rmd \mu.
 \end{align}
We used that $\mu - \pi_{r,f} \mu$ is in the constraint space $U^{d-k-1}(f)$, and that likewise  $\mu - \pi_{r,f'} \mu \in U^{d'-k}(f')$, to insert such terms in (\ref{eq:beforein}).  
 
For $\nu \in \kappa \Poly^{r-1} \Lambda^{d-k+1}(f)$ we have:
\begin{align}
\int_f P_{r,f} \omega \wedge \nu & = \int_f \hs^{-1} \omega_f \wedge \nu,\\
& =  \int_{f} u \wedge \nu = \int_f \gamma_f u \wedge \nu.  
\end{align}

\noindent Using \cite[Eq. (2.9)]{BonDiPDroHu25} (based on \cite[\S 3.2]{ArnFalWin06}) viz.:
\begin{equation}
  \Poly^{r} \Lambda^{d-k}(f) =  \rmd \kappa \Poly^{r} \Lambda^{d-k}(f) \oplus  \kappa \Poly^{r-1} \Lambda^{d-k+1}(f),
\end{equation}
we conclude.
\end{proof}

\section*{Declaration}
On 22 November 2026, we, the authors, had a conversation about arXiv version 3 of this paper with an AI driven chatbot (Claude Sonnet 5.5).

It pointed out that in the analysis in \S 5, the used high order regularity of the elements of $A^k(T)$ does not hold in general. We have here corrected this issue by introducing weights in the Sobolev seminorms.

Furthermore the chatbot spontaneously provided a proof that the parallells between FES and DDR that we had pointed out in Remark 5.1 can be made entirely precise. The exact correspondence is here included in \S \ref{sec:cor}. Proposition \ref{prop:corfesddr} was proved by the chatbot and edited by us for readability.

We stress that before 22 November 2026 we had not used AI in any significant way and that these two points are the only parts of the final manuscript that were assisted by AI.

}

\bibliography{../Bibliography/alexandria,../Bibliography/newalexandria,../Bibliography/mybibliography}{}
\bibliographystyle{amsplain}

\end{document}